\documentclass[12pt,english]{article}
\usepackage[margin=1.2in]{geometry}
\usepackage[T1]{fontenc}
\usepackage[latin9]{inputenc}
\usepackage{color}
\usepackage{amsmath}
\usepackage{amsthm}
\usepackage{amssymb}
\PassOptionsToPackage{normalem}{ulem}
\usepackage{ulem}

\makeatletter

\numberwithin{figure}{section}
\numberwithin{equation}{section}
\theoremstyle{definition}
\newtheorem{example}{\protect\examplename}[section]
\theoremstyle{plain}
\newtheorem{lem}{\protect\lemmaname}[section]
\ifx\proof\undefined\
  \newenvironment{proof}[1][\proofname]{\par
    \normalfont\topsep6\p@\@plus6\p@\relax
    \trivlist
    \itemindent\parindent
    \item[\hskip\labelsep
          \scshape
      #1]\ignorespaces
  }{%
    \endtrivlist\@endpefalse
  }
  \providecommand{\proofname}{Proof}
\fi
\theoremstyle{remark}
\newtheorem{rem}{\protect\remarkname}[section]
\theoremstyle{plain}
\newtheorem{thm}{\protect\theoremname}[section]
\theoremstyle{plain}
\newtheorem{prop}{\protect\propositionname}[section]
\theoremstyle{plain}
\newtheorem{cor}{\protect\corollaryname}[section]
\theoremstyle{definition}
\newtheorem{defn}{\protect\definitionname}[section]
\theoremstyle{definition}
\newtheorem*{abGEthm}{\protect\abGEthmname}

\usepackage[title]{appendix}
\usepackage{comment}
\usepackage{graphicx}
\usepackage{float}
\usepackage{hyperref}

\hypersetup{
     colorlinks   = true,
     citecolor    = blue,
     linkcolor    = blue
}

\date{}

\makeatother

\usepackage{babel}
\providecommand{\corollaryname}{Corollary}
\providecommand{\definitionname}{Definition}
\providecommand{\examplename}{Example}
\providecommand{\lemmaname}{Lemma}
\providecommand{\propositionname}{Proposition}
\providecommand{\remarkname}{Remark}
\providecommand{\theoremname}{Theorem}
\providecommand{\abGEthmname}{Abstract G\"artner-Ellis Theorem}

\usepackage{amsmath,amssymb}
\newcommand{\iii}[1]{{\left\vert\kern-0.25ex\left\vert\kern-0.25ex\left\vert #1 
    \right\vert\kern-0.25ex\right\vert\kern-0.25ex\right\vert}}

\newcommand{\EE}{\mathbb{E}}     
\newcommand{\HH}{\mathbb{H}}

\newcommand{\RR}{\mathbb{R}}      

\newcommand{\ZZ}{\mathbb{Z}}      

\newcommand{\bB}{\mathcal{B}}

\renewcommand{\d}{\partial}

\newcommand{\eps}{\varepsilon}

\begin{document}
\title{Parabolic Anderson Model in the Hyperbolic
Space. Part I: Annealed Asymptotics}
\author{Xi Geng\thanks{School of Mathematics and Statistics, University of Melbourne, Australia.
Email: xi.geng@unimelb.edu.au.}$\ $ and Weijun Xu\thanks{Beijing International Center for Mathematical Research, Peking University,
China. Email:weijunxu@bicmr.pku.edu.cn.}}
\maketitle
\begin{abstract}
We establish the second-order moment asymptotics
for a parabolic Anderson model $\partial_{t}u=(\Delta+\xi)u$ in the
hyperbolic space with a regular, stationary Gaussian potential
$\xi$. It turns out that the growth and fluctuation asymptotics both are identical to the Euclidean situation. As a result, the solution exhibits the same moment intermittency property as in the Euclidean case. An interesting point here is that the fluctuation exponent is determined by a variational problem induced by the Euclidean (rather than hyperbolic) Laplacian. Heuristically, this is due to a curvature dilation effect:
the geometry becomes asymptotically flat after suitable renormalisation
in the derivation of the second-order asymptotics. 

\tableofcontents{}
\end{abstract}

\section{Introduction}

This article is the first part of our investigation on the long time behaviour of the solution $u$ to the parabolic Anderson model (PAM)
\begin{equation} \label{eq:PAM}
    \d_t u = \Delta u + \xi u\;, \qquad u(0,\cdot) \equiv 1
\end{equation}
on the $d$-dimensional space-form $M = \HH^d$ with constant negative curvature $\kappa \equiv -1$. Here $\xi$ is a stationary mean-zero Gaussian field on $M$ with smooth correlation and $\Delta$ is the Laplace-Beltrami operator on $M$. Our basic motivation is to look for new asymptotic behaviour and mechanism that arises from effects of non-Euclidean geometry.

Among several fundamental questions surrounding the study of PAM, our main goal is to understand the moment (annealed) and almost sure (quenched) asymptotics of the solution to (\ref{eq:PAM}) in large time. The present article studies the annealed situation and the second part \cite{GX25} (Part II: Quenched Asymptotics) deals with the quenched situation. Below is a brief summary of our main findings:

\vspace{2mm}\noindent\textbf{Annealed asymptotics}. The moment asymptotics turns out to be identical to the Euclidean situation. In particular, the solution exhibits the same moment intermittency property as in the Euclidean case. A surprising point here is that the fluctuation exponent is determined by a variational problem associated with the Euclidean (rather than hyperbolic) Laplacian. 

\vspace{2mm}\noindent\textbf{Quenched asymptotics}. The growth rate and the exact growth constant are determined through an explicit optimisation procedure, which are both different from the Euclidean situation. The almost sure growth is much faster in the hyperbolic case ($t^{5/3}$ vs $t\sqrt{\log t}$) in logarithmic scale) and the solution exhibits a stronger non-Euclidean localisation mechanism. 

\vspace{2mm}\noindent In summary, the annealed asymptotics is  insensitive to the underlying geometry while negative curvature (more essentially, the exponential volume growth) plays a crucial role in the hyperbolic quenched asymptotics.

\vspace{2mm}\textit{Annealed asymptotics and intermittency}. For every $p \geqslant 1$, let us denote
\begin{equation*}
    m_p(t) \triangleq \EE \big[ u(t,x)^p \big]\;.
\end{equation*}
This quantity does not depend on $x$ by the stationarity of $\xi$. In the present article, we aim at investigating the behaviour of $m_p(t)$ for large $t$. The long time behaviour of $m_p(t)$ has been extensively investigated in the cases of $\RR^d$ or $\ZZ^d$ for various random potentials $\xi$ (cf. \cite{K\"on16, GK00} and references therein). An important motivation for studying the moment asymptotics is to understand the intermittency behaviour of the solution, namely, the main contribution to moments comes from high peaks that are sparsely distributed in space as $t \rightarrow +\infty$. Mathematically, this can be captured by showing that for every $p>q$, one has
\begin{equation} \label{eq:intermittency_log_moment}
    \frac{1}{p} \log m_p(t) - \frac{1}{q} \log m_q(t) \gg 0
\end{equation}
as $t \rightarrow +\infty$. We refer the reader to \cite{GM90, K\"on16} for a discussion on intermittency. It was shown in \cite{GK00} that for a large class of random fields $\xi$, one has the asymptotic expansion
\begin{equation}\label{eq:MomGen}
    \log m_p(t) = H(pt) - \beta (pt) \big( \chi + o_t(1) \big)
\end{equation}
to the second order, where $H$ is the cumulant generating function of the random variable $\xi(0)$, $\beta$ is a scaling function related to the law of $\xi$ and satisfies
\begin{equation*}
    1 \ll \beta(t) \ll H(t)\; \qquad \text{as} \; t \rightarrow +\infty\;,
\end{equation*}
and $\chi \in \RR$ is a constant determined by the Laplacian and the law of $\xi$. In particular, the moment intermittency property (\ref{eq:intermittency_log_moment}) holds for a rich class of $\xi$ including Gaussian potentials.

\vspace{2mm}\textit{Main result}. In this article, we show that the moment asymptotics for the solution to the PAM (\ref{eq:PAM}) under a stationary, regular Gaussian potential is identical to the Euclidean situation (\ref{eq:MomGen}) (cf. Theorem \ref{thm:MainThm} below for the precise formulation). In particular, the solution exhibits the same moment intermittency property as in the Euclidean case. A surprising point here is that the fluctuation exponent $\chi$ is also determined by the Euclidean (rather than hyperbolic Laplacian). This fact is a priori unclear and a heuristic explanation is given in Section \ref{sec:EucFluc}. We will discuss our main strategy for proving Theorem \ref{thm:MainThm} and the underlying difficulties in Section \ref{sec:Strat}. 

\vspace{2mm}\textit{Related works in the time-dependent case}. There are several related works that investigate the case of a time-dependent potential in a geometric setting. For instance, \cite{BCO24} considered the PAM with a time-dependent (white in time and coloured in space) Gaussian potential on a Cartan-Hadamard manfold, where the authors obtained well-posedness results as well as moment estimates for the solution. The role of non-positive curvature (and global geometry) also plays an essential role in the recent work \cite{CO25} where similar questions as in \cite{BCO24} were investigated. Other works related to the asymptotic behaviour of PAM with a time-dependent Gaussian potential in a geometric setting include e.g. \cite{TV02, BOT23, BCH24, COV24}. We should point out that the case of time-dependent potential is  different from the time-independent case in several fundamental ways (the asymptotic results, methodology and underlying mechanism have very different nature in these two settings). Here we focus on the time-independent case and will not comment much on the other situation.

\section{Basic notions and statement of the main theorem}

In this section, we discuss the problem set-up and state the main
theorem of the article.

\subsection{The parabolic Anderson model in the hyperbolic space}

Let $M=\mathbb{H}^{d}$ be the $d$-dimensional, complete and simply-connected Riemannian manifold of curvature $\kappa\equiv-1$. Let
$\xi=\{\xi(x):x\in M\}$ be a given fixed mean zero, stationary Gaussian
field on $M$, namely, we assume that 

\vspace{2mm}\noindent (A1) $\mathbb{E}[\xi(x)]=0$ for all $x\in M$;

\vspace{1mm}\noindent (A2) $\xi(g\cdot)\stackrel{d}{=}\xi(\cdot)$
for all orientation-preserving isometries $g$ of $M$.

\vspace{2mm}\noindent One can show that (cf. Lemma \ref{lem:CovRad}
below) the \textit{covariance function} of $\xi$ defined by 
\[
C(x,y)\triangleq\mathbb{E}[\xi(x)\xi(y)],\ \ \ x,y\in M
\]
is a function of hyperbolic distance $d(x,y)$. We therefore write
$C(x,y)=Q(d(x,y))$. We make one further assumption on $Q$:

\vspace{2mm}\noindent (A3) $Q$ is twice continuously differentiable.

\vspace{2mm}Typical examples of $\xi$ are given by Example \ref{exa:GField}
below. In this article, we consider the following \textit{parabolic
Anderson model} (PAM) with respect to the above Gaussian field $\xi$
on $M$:
\begin{equation}
\begin{cases}
\partial_{t}u(t,x)=\Delta u(t,x)+\xi(x)u(t,x), & (t,x)\in(0,\infty)\times M;\\
u(0,\cdot)\equiv1.
\end{cases}\label{eq:GlobalPAM}
\end{equation}
Here $\Delta$ denotes the Laplace-Beltrami operator on $M$. Our
goal is to compute the exact asymptotics of the $p$-th moment ($p\geqslant1$)
of $u(t,x)$ as $t\rightarrow\infty$ (both leading growth and fluctuation).

\subsection{Some geometric notions and the hyperboloid model}

We will frequently use geodesic polar coordinates in our analysis
which we shall first recall. Let $o\in M$ be a given fixed base point.
It is well-known that the exponential map $\exp_{o}:T_{o}M\rightarrow M$
is a global diffeomorphism. In particular, any point $x\neq o$ on
$M$ is uniquely written as $x=\exp_{o}(\rho\sigma)$ with $\rho>0$
(the radial component) and $\sigma\in ST_{o}(M)\triangleq\{v\in T_{o}M:|v|=1\}$
(the angular component). Of course $o$ corresponds to $\rho=0.$
The parametrisation 
\[
\exp_{o}:(\rho,\sigma)\mapsto x=\exp_{o}(\rho\sigma)
\]
is known as the \textit{geodesic polar chart} with respect to $o.$
Given $\lambda>0$ and $x\in M$, we define the \textit{dilation}
with respect to $o$
\begin{equation}
\lambda\cdot x\triangleq\exp_{o}(\lambda\exp_{o}^{-1}x).\label{eq:Dilation}
\end{equation}
Using geodesic polar coordinates, this is simply $(\rho,\sigma)\mapsto(\lambda\rho,\sigma)$.

Under geodesic polar chart, the metric tensor, Laplacian and volume
form are respectively given by 
\begin{align}
 & ds^{2}=d\rho^{2}+\sinh^{2}\rho d\sigma^{2},\label{eq:RMetric}\\
 & \Delta=\partial_{\rho}^{2}+\coth\rho\partial_{\rho}+\sinh^{-2}\rho\Delta_{\sigma},\label{eq:DeltaPolar}\\
 & {\rm vol}(d\rho,d\sigma)=\sinh^{d-1}\rho d\rho{\rm vol}(d\sigma).\nonumber 
\end{align}
Here $d\sigma^{2}$ is the metric tensor, $\Delta_{\sigma}$ is the
Laplacian and ${\rm vol}(d\sigma)$ is the volume form on the unit
sphere $ST_{o}M$. We use $\omega_{d-1}$ to denote the spherical
volume of $ST_{o}M.$ More generally, if the underlying curvature
is $\kappa\equiv-\alpha^{2}$ ($\alpha>0$), then under geodesic polar
coordinates one has 
\begin{align}
 & ds^{2}=d\rho^{2}+\alpha^{-2}\sinh^{2}(\alpha\rho)d\sigma^{2},\label{eq:MetAlp}\\
 & \Delta=\partial_{\rho}^{2}+\alpha\coth(\alpha\rho)\partial_{\rho}+\alpha^{2}\sinh^{-2}(\alpha\rho)\Delta_{\sigma},\label{eq:LapAlp}\\
 & {\rm vol}(d\rho,d\sigma)=\alpha^{-(d-1)}\sinh^{d-1}(\alpha\rho)d\rho{\rm vol}(d\sigma).\label{eq:VolAlp}
\end{align}
The hyperbolic distance function on $M$ is denoted as $d(\cdot,\cdot).$
The group of orientation-preserving isometries is denoted as $G.$
For each $R>0,$ we denote $Q_{R}\triangleq\{x\in M:d(x,o)\leqslant R\}$.

\vspace{2mm} A convenient model of $M$ to perform explicit analysis
is the hyperboloid model which is described as follows. For each $\alpha>0,$
let us define 
\begin{equation}
\mathbb{H}_{\alpha}^{d}\triangleq\{x=(x_{1},\cdots,x_{d+1})^{T}:x*x=-\alpha^{-2},x_{d+1}>0\}\subseteq\mathbb{R}^{d+1},\label{eq:Hyperboloid}
\end{equation}
where $*$ is the Lorentzian inner product on $\mathbb{R}^{d+1}$
defined by 
\[
x*y\triangleq\sum_{i=1}^{d}x_{i}y_{i}-x_{d+1}y_{d+1}.
\]
The Lorentzian metric restricted to $\mathbb{H}_{\alpha}^{d}$ becomes
positive definite, turning $\mathbb{H}_{\alpha}^{d}$ into a $d$-dimensional,
complete and simply-connected Riemannian manifold of curvature $\kappa\equiv-\alpha^{2}.$
The space $\mathbb{H}_{1}^{d}$ is the \textit{hyperboloid model }of
$M$. The Riemannian distance function $d^{\alpha}$ on $\mathbb{H}_{\alpha}^{d}$
is explicitly determined through the following relation:
\begin{equation}
x*y=-\alpha{}^{-2}\cosh\big(\alpha d^{\alpha}(x,y)\big),\ \ \ x,y\in\mathbb{H}_{\alpha}^{d}.\label{eq:HypDist}
\end{equation}
A distinguished base point of $\mathbb{H}_{\alpha}^{d}$ is chosen
to be $o=(0,0,\alpha^{-1}).$ Geodesics passing through $o$ are precisely
those intersection curves between $\mathbb{H}_{\alpha}^{d}$ and any
two-dimensional subspace of $\mathbb{R}^{d+1}$ containing the $x_{d+1}$-axis. 

The group of orientation-preserving isometries for $\mathbb{H}_{\alpha}^{d}$
is 
\[
G={\rm SO}^{+}(d,1)=\big\{ g\in{\rm Mat}_{d+1}(\mathbb{R}):g^{T}Jg=J,A_{d+1}^{d+1}>0\big\},
\]
where $J\triangleq{\rm diag}(1,\cdots,1,-1)$. Here the $G$-action
is just the standard linear action; in fact $G$ acts simultaneously
on every $\mathbb{H}_{\alpha}^{d}$ and is the group of orientation-preserving
isometries for all these spaces. The isotropy group $K$ with respect
to the base point $o$ (i.e. the subgroup of $G$ leaving $o$ fixed)
is the orthogonal group: 
\[
K=\Big\{ k\in G:k=\left(\begin{array}{cc}
S & 0\\
0 & 1
\end{array}\right)\ \text{with }S\in{\rm SO}(d)\Big\}.
\]
Note that $k\in K$ acts on each geodesic sphere 
\begin{align*}
S_{r} & \triangleq\{x\in\mathbb{H}_{1}^{d}:d(x,o)=r\}\\
 & =\big\{(x_{1},\cdots,x_{d+1}):x_{1}^{2}+\cdots+x_{d}^{2}=\sinh^{2}r,\ x_{d+1}=\cosh r\big\}
\end{align*}
by spherically rotating the $(x_{1},\cdots,x_{d})^{T}$-component
by the orthogonal matrix $S$. It is well-known that these actions
are transitive: for any $x,y\in\mathbb{H}_{1}^{d}$ (respectively,
$x,y\in S_{r}$) there exists $g\in G$ (respectively, $g\in K$)
such that $y=gx$.

\subsection{Basic properties and a key scaling limit of the Gaussian field}

Here we collect some basic facts about the Gaussian field $\xi$.
We also derive an important functional arising from suitable scaling
limit of $\xi$, which will appear explicitly in the fluctuation exponent
of the PAM (cf. (\ref{eq:FlucExp})). 

To justify its interest, we must first demonstrate the existence of
rich examples of Gaussian fields on $M$ satisfying Assumptions (A1)$\sim$(A3). 
\begin{example}
\label{exa:GField}Let $\Xi(\cdot)$ denote the white noise on the
isometry group $G$. In other words, $\{\Xi(A):A\in{\cal B}(G),|A|_{G}<\infty\}$
is a Gaussian family with mean zero and covariance structure 
\[
\mathbb{E}[\Xi(A)\Xi(B)]=|A\cap B|_{G}.
\]
Here $|A|_{G}\triangleq\int_{A}dg$ where $dg$ is a fixed Haar measure
on $G$. Let $f$ be a given smooth function on $M$ with suitable
decay at infinity. Then 
\[
\xi(x)\triangleq\int_{G}f(g^{-1}\cdot x)\Xi(dg),\ \ \ x\in M
\]
defines a smooth Gaussian field on $M$ with covariance function 
\[
C(x,y)=\int_{G}f(g^{-1}\cdot x)f(g^{-1}\cdot y)dg=\int_{G}f(g\cdot x)f(g\cdot y)dg,
\]
where $g\cdot x$ means the action of $g\in G$ on $x\in M$. It follows
that 
\begin{align*}
C(h\cdot x,h\cdot y) & =\int_{G}f(g\cdot(h\cdot x))f(g\cdot(h\cdot y))dg\\
 & \stackrel{l=gh}{=}\int_{G}f(l\cdot x)f(l\cdot y)dl=C(x,y)
\end{align*}
for all $h\in G$ and $x,y\in G.$ This shows that the distribution
of $\xi$ is invariant under $G$-actions. 
\end{example}
\begin{lem}
\label{lem:CovRad}$C(x,y)$ is a function of $d(x,y)$.
\end{lem}
\begin{proof}
Let $x_{i},y_{i}\in M$ ($i=1,2$) be such that $d(x_{1},y_{1})=d(x_{2},y_{2})$.
Let $g_{i}$ be an isometry such that $g_{i}x_{i}=o$ and set $y_{i}'\triangleq g_{i}y_{i}.$
Since the Gaussian field $\xi$ is stationary, one has 
\[
C(x_{1},y_{1})=C(g_{1}x_{1},g_{1}y_{1})=C(o,y_{1}')
\]
and similarly $C(x_{2},y_{2})=C(o,y_{2}').$ By assumption, one has
$d(o,y_{1}')=d(o,y_{2}')$. As a result, there is a rotation $k\in K$
such that $y_{2}'=ky_{1}'$ and $o=ko$. It follows that 
\[
C(o,y'_{2})=C(ko,ky_{1}')=C(o,y_{1}')\implies C(x_{1},y_{1})=C(x_{2},y_{2}).
\]
This shows that $C(x,y)$ is a function of distance only. 
\end{proof}
As a consequence of Lemma \ref{lem:CovRad}, one can define $Q(r)\triangleq C(x,y)$
for $r=d(x,y)$. We denote 
\[
\sigma^{2}\triangleq Q(0)={\rm Var}[\xi(x)].
\]
It is useful to note that $|Q(r)|\leqslant Q(0)$ for all $r\geqslant0,$
which is easily seen by the fact that 
\[
\mathbb{E}[(\xi(z)\pm\xi(0))^{2}]=2Q(0)\pm2Q(d(z,o))\geqslant0.
\]
In particular, $Q'(0)=0$ and $Q''(0)\leqslant0$.

\subsubsection*{A weak scaling limit of $\xi$}

The Gaussian field $\xi$ possesses a weak scaling limit which plays
a key role in deriving the fluctuation exponent of the PAM. Throughout
the result of the article, we define the cumulant generating function
\begin{equation}
H(t)\triangleq\log\langle e^{t\xi(z)}\rangle=\frac{1}{2}\sigma^{2}t^{2}\label{eq:CGF}
\end{equation}
and fix two basic scaling parameters: 
\begin{equation}
\alpha(t)\triangleq t^{-1/4},\ \beta(t)\triangleq\frac{t}{\alpha^{2}(t)}=t^{3/2}.\label{eq:BetaScale}
\end{equation}
We define the \textit{rescaled Gaussian field} $\xi_{t}$ by 
\begin{equation}
\xi_{t}(x)\triangleq\alpha^{2}(t)\big(\xi(\alpha(t)\cdot x)-\frac{H(t)}{t}\big),\ \ \ t>0,x\in M,\label{eq:RescaledXi}
\end{equation}
where $\alpha(t)\cdot x$ is the dilation defined by (\ref{eq:Dilation}).

It is convenient to think of $\xi_{t}$ as a field on the space of
curvature $\kappa_{t}\equiv-\alpha(t)^{2}$. Under the hyperboloid
model $\mathbb{H}_{\alpha(t)}^{d}$, for each $w\in\mathbb{H}_{\alpha(t)}^{d}$
let $x_{w}=\alpha(t)^{-1}\cdot(\alpha(t)\times w)$ be the corresponding
element on $M=\mathbb{H}_{1}^{d}$, where $\times$ is the actual
scalar multiplication. Then one can identify the field $\xi_{t}$
with $\hat{\xi}_{t}(w)\triangleq\xi_{t}(x_{w})$ on $\mathbb{H}_{\alpha(t)}^{d}$.
\begin{lem}
\label{lem:IsoInv}Under the above identification, the distribution
of the Gaussian field $\xi_{t}$ is invariant under isometries. 
\end{lem}
\begin{proof}
Under the hyperboloid model, one has 
\[
\hat{\xi}_{t}(w)=\alpha(t)^{2}\big[\xi(\alpha(t)\times w)-\frac{H(t)}{t}\big],\ \ \ w\in\mathbb{H}_{\alpha(t)}^{d}.
\]
Recall that the group $G={\rm SO}^{+}(d,1)$ acts by isometries simultaneously
on all of the $\mathbb{H}_{\alpha(t)}^{d}$'s. Let $g\in G$ be a
given isometry. It follows from the stationarity of $\xi$ that 
\begin{align*}
\hat{\xi}_{t}(g\times\cdot) & =\alpha(t)^{2}\big[\xi\big(\alpha(t){\rm Id}\times(g\times\cdot)\big)-\frac{H(t)}{t}\big]\\
 & =\alpha(t)^{2}\big[\xi\big(g\times(\alpha(t){\rm Id}\times\cdot)\big)-\frac{H(t)}{t}\big]\\
 & \stackrel{d}{=}\alpha(t)^{2}\big[\xi\big((\alpha(t){\rm Id}\times\cdot\big)-\frac{H(t)}{t}\big]=\hat{\xi}_{t}(\cdot).
\end{align*}
The result thus follows. 
\end{proof}
Given any compact subset $K\subseteq M$, we use ${\cal P}(K)$ to
denote the set of probability measures supported on $K$ (equivalently,
$\mu(K)=1$). Let ${\cal P}_{c}(M)$ denote the space of compactly
supported probability measures on $M$. We say that $\mu_{t}\Rightarrow\mu$
\textit{in} ${\cal P}_{c}(M)$ as $t\rightarrow\infty$ if there exists
a compact set $K\subseteq M$ such that $\mu_{t},\mu\in{\cal P}(K)$
for all $t$ and 
\[
\lim_{t\rightarrow\infty}\mu_{t}=\mu\ \text{weakly}.
\]
Given $\mu\in{\cal P}_{c}(M)$ and $f:M\rightarrow\mathbb{R}$, we
write 
\[
(\mu,f)\triangleq\int_{M}f(x)\mu(dx).
\]
We define a functional $J_{t}:{\cal P}_{c}(M)\rightarrow\mathbb{R}$
by
\begin{equation}
J_{t}(\mu)\triangleq-\frac{1}{\beta(t)}\log\langle e^{\beta(t)(\mu,\xi_{t})}\rangle,\ t>0.\label{eq:JtDef}
\end{equation}
Simple applications of Jensen's and H\"older's inequalities show
that $J_{t}$ is non-negative and concave. 

The following lemma is crucial for the derivation of the fluctuation
asymptotics. Given $z,w\in M$ with polar coordinates $(\rho_{1},\sigma_{1}),$
$(\rho_{2},\sigma_{2})$ with respect to $o$, we define the distance function $d_{{\rm eu}}$ by 
\[
d^2_{{\rm eu}}(z,w)\triangleq\rho_{1}^{2}+\rho_{2}^{2}-2\rho_{1}\rho_{2}\langle\sigma_{1},\sigma_{2}\rangle_{T_{o}M}.
\]
Note that this is just the Euclidean distance between $z$ and $w$
by treating $(\rho_{i},\sigma_{i})$ as their Euclidean polar coordinates. 
\begin{lem}
\label{lem:JtConv}Suppose that $Q(r)$ is twice continuously differentiable.
Then $J_{t}(\mu)$ converges to\textcolor{black}{{} the functional
\begin{equation}
J:{\cal P}_{c}(M)\rightarrow[0,\infty),\ J(\mu)\triangleq-\frac{1}{4}Q''(0)\int_{M\times M}d_{{\rm eu}}(z,w)^{2}\mu(dz)\mu(dw)\label{eq:JFunct}
\end{equation}
}as $t\rightarrow\infty$ uniformly for all $\mu\in{\cal P}(K)$ with
$K$ being any given compact subset of $M$. 
\end{lem}
\begin{proof}
Observe that 
\[
(\mu,\xi(\alpha(t)\cdot))=\int_{M}\xi(\alpha(t)\cdot x)\mu(dx)
\]
is a mean zero Gaussian variable with variance 
\[
{\rm Var}[(\mu,\xi(\alpha(t)\cdot))]=\int_{M\times M}Q(d(\alpha(t)\cdot z,\alpha(t)\cdot w))\mu(dz)\mu(dw).
\]
According to the definition (\ref{eq:RescaledXi}) of $\xi_{t}$ and
the basic scaling relations (\ref{eq:BetaScale}), one finds that
\begin{align}
\log\langle e^{\beta(t)(\mu,\xi_{t})}\rangle= & \frac{t^{2}}{2}\int_{M\times M}\big(Q(d(\alpha(t)\cdot z,\alpha(t)\cdot w))-Q(0)\big)\mu(dz)\mu(dw)\nonumber \\
= & \frac{t^{2}}{2}\int_{0}^{1}(1-\theta)d\theta\int_{M\times M}Q''(\theta d(\alpha(t)\cdot z,\alpha(t)\cdot w))\nonumber \\
 & \ \ \ \times d(\alpha(t)\cdot z,\alpha(t)\cdot w)^{2}\mu(dz)\mu(dw)\;,\label{eq:JFuncPf4}
\end{align}
where we have used that $Q'(0)=0$. Under the hyperboloid model $\mathbb{H}_{1}^{d}$, by writing $z=(\rho_{1},\sigma_{1})$ and $w=(\rho_{2},\sigma_{2})$ and performing explicit calculation, one has
\begin{align} \label{eq:JFuncPf1}
    \cosh d(\alpha(t)\cdot z,\alpha(t)\cdot w)-1 = \frac{1}{2}\alpha(t)^{2}d_{{\rm eu}}(z,w)^{2}+O(\alpha(t)^{4})
\end{align}
as $t\rightarrow\infty$. Meanwhile, one also knows that 
\begin{equation}
\cosh d(\alpha(t)\cdot z,\alpha(t)\cdot w)-1=\frac{1}{2}d(\alpha(t)\cdot z,\alpha(t)\cdot w)^{2}+O\big(d(\alpha(t)\cdot z,\alpha(t)\cdot w)^{4}\big).\label{eq:JFuncPf2}
\end{equation}
By equating (\ref{eq:JFuncPf1}) and (\ref{eq:JFuncPf2}), it is immediate
to see that 
\begin{equation}
d(\alpha(t)\cdot z,\alpha(t)\cdot w)=\alpha(t)d_{{\rm eu}}(z,w)+O(\alpha(t)^{2}).\label{eq:JFuncPf3}
\end{equation}
After substituting (\ref{eq:JFuncPf3}) back into (\ref{eq:JFuncPf4}),
one obtains that 
\begin{equation}
\log\langle e^{\beta(t)(\mu,\xi_{t})}\rangle=\big(\frac{\alpha(t)^{2}}{4}Q''(0)\int_{M\times M}d_{{\rm eu}}(z,w)^{2}\mu(dz)\mu(dw)+O(\alpha(t)^{4})\big)t^{2}.\label{eq:JFuncPf5}
\end{equation}
The result thus follows by making use of the explicit relations that
$\alpha(t)=t^{-1/4}$ and $\beta(t)=t^{3/2}$ (cf. (\ref{eq:BetaScale})).
All the above expansions are easily seen to be uniform with respect
to $z,w\in K$. 
\end{proof}
\begin{rem}
The functional $J$ does not ``feel'' the geometry of $M$ and essentially
lives in Euclidean geometry. In fact, from its expression (\ref{eq:JFunct})
one could just regard $M\cong\mathbb{R}^{n}$ with $d_{{\rm eu}}(z,w)$
being the usual Euclidean distance. 
\end{rem}

\subsection{Statement of the main result}

The main result of the present article is stated as follows.
\begin{thm}
\label{thm:MainThm}For any $p\geqslant1$, the following second order
asymptotics holds true for the solution $u(t,x)$ to the PAM (\ref{eq:GlobalPAM}):
\begin{equation}
\lim_{t\rightarrow\infty}\frac{1}{\beta(pt)}\log\big(e^{-H(pt)}\langle u(t,o)^{p}\rangle\big)=-\sqrt{-\frac{Q''(0)}{2}}d,\label{eq:MainAsym}
\end{equation}
where $H(t)$, $\beta(t)$ are defined by (\ref{eq:CGF}), (\ref{eq:BetaScale})
respectively and $\langle\cdot\rangle$ means taking expectation with
respect to the $\xi$-randomness. In other words, the $p$-th moment
of $u(t,o)$ satisfies the asymptotics
\begin{equation}
\langle u(t,o)^{p}\rangle=\exp\Big(\frac{1}{2}\sigma^{2}p^{2}t^{2}-d\sqrt{-\frac{Q''(0)}{2}}p^{3/2}t^{3/2}(1+o(1))\Big)\ \ \ \text{as }t\rightarrow\infty.\label{eq:FlucAsym}
\end{equation}
\end{thm}

\subsubsection{The Euclidean nature of the fluctuation exponent}\label{sec:EucFluc}

We make an important comment on the fluctuation exponent given in
(\ref{eq:MainAsym}). Let us first introduce the so-called \textit{Donsker-Varadhan
functional} on the \textit{Euclidean} space $T_{o}M$. This is the
functional ${\cal S}_{{\rm eu}}:{\cal P}_{c}(T_{o}M)\rightarrow[0,+\infty]$
defined by 
\begin{equation}
\mathcal{S}_{{\rm eu}}(\nu)\triangleq\begin{cases}
\int_{T_{o}M}|\nabla\phi|^{2}dx, & \nu\ll dx,\dfrac{d\nu}{dx}=\phi^{2}\text{ with some }\phi\in H^{1}(T_{o}M);\\
+\infty, & \text{otherwise}.
\end{cases}\label{eq:DVFunc}
\end{equation}
Here $dx$ is the Lebesgue measure on $T_{o}M$ and $H^{1}(T_{o}M)$
is the standard $(1,2)$-Sobolev space over $T_{o}M$. Note that if
$\nu$ is supported on some bounded domain $B$ and $\mathcal{S}_{{\rm eu}}(\nu)<\infty,$
then $\sqrt{d\nu/dx}\in H_{0}^{1}(B)$ ($H^1$ functions
vanishing at the boundary $\partial B$). 

What we will actually prove is that 
\begin{equation}
\lim_{t\rightarrow\infty}\frac{1}{\beta(pt)}\log\big(e^{-H(pt)}\langle u(t,o)^{p}\rangle\big)=-\chi\label{eq:LpAsym}
\end{equation}
where 
\begin{equation}
\chi\triangleq\inf\big\{ J(\mu)+{\cal S}_{{\rm eu}}(\exp_{o}^{-1}\mu):\mu\in{\cal P}_{c}(M)\big\}.\label{eq:FlucExp}
\end{equation}
A remarkable point is that although we are considering the PAM in the hyperbolic space $M$, the fluctuation exponent $\chi$ given by (\ref{eq:FlucExp}) does \textit{not} ``feel'' the non-Euclidean
geometry of $M$. Indeed, it is clear from the expression (\ref{eq:JFunct})
that the functional $J$ also has an Euclidean nature; it is essentially
the same limiting functional $J$ arising in the Euclidean PAM considered
in \cite{GK00} when the Gaussian field is invariant under Euclidean
isometries. As a result, one can directly apply \cite[Identity (4.6)]{GK00}
to conclude that $\chi=\sqrt{-Q''(0)/2}d$ (cf. Lemma \ref{lem:ChiValue}
 below).

Heuristically, the fundamental reason behind this is that the fluctuation
asymptotics arises from a suitable rescaling of the PAM as well as
of the Gaussian field. However, when such rescaling is applied the
underlying geometry goes through a curvature dilation effect: the
curvature is rescaled from $\kappa\equiv1$ to $\kappa_{t}\equiv-\alpha(t)^{2}$.
Since $\alpha(t)=t^{-1/4}\rightarrow0$ as $t\rightarrow\infty$,
the limiting geometry where the fluctuation asymptotics takes place
is thus flat. Making such heuristics precise requires substantial
effort. The main difficulty is that one cannot simultaneously rescale
the geometry of $M$ to Euclidean and send the time parameter of the
rescaled PAM to infinity on the non-compact state space in one go.
One needs to perform localisation analysis over compact domains in
a careful way.

\subsubsection{Outline of main strategy and difficulties}\label{sec:Strat}

To prove Theorem \ref{thm:MainThm}, we will follow the main strategy
developed in \cite{GK00}. However, there are several challenges to
overcome in the hyperbolic situation. We first outline the main steps
and point out the underlying difficulties. 

The essence of the G\"artner-K\"onig argument is quite natural to
describe (we only consider the case when $p=1$). First of all, after
applying a suitable rescaling and re-centering of the PAM (cf. \eqref{eq:RescaleLocalPAM} below) and the Feynman-Kac representation (\ref{eq:FKRes}), one finds
that 
\[
e^{-H(t)}u(t,o)=\mathbb{E}_{o}\big[e^{\int_{0}^{\beta(t)}\xi_{t}(W_{s}^{t})ds}\big],
\]
where $W^{t}$ is now a Brownian motion in the hyperbolic space of
curvature $\kappa_{t}\equiv-\alpha(t)^{2}$. Of course, this is the
correct normalisation in view of the first order asymptotics derived in
Proposition \ref{prop:1stOrdAsym}. In terms of the functional $J_{t}$ (cf. \eqref{eq:JtDef}), one can now write 
\[
e^{-H(t)}\langle u(t,o)\rangle=\mathbb{E}_{o}\big[e^{-\beta(t)J_{t}(L_{\beta(t)}^{t})}\big],
\]
where 
\begin{equation}
L_{s}^{t}(dx)\triangleq\frac{1}{s}\int_{0}^{s}{\bf 1}_{\{W_{r}^{t}\in dx\}}dr,\ \ \ s>0\label{eq:OTMP}
\end{equation}
is the \textit{occupation time measure process} for $W^{t}.$ The key
point is that the family of measures $\{{\rm Law}(L_{\beta(t)}^{t}):t>0\}$
satisfies a large deviation principle (LDP) with rate function ${\cal S}_{{\rm eu}}$
(the Donsker-Varadhan LDP). Since $J_{t}\rightarrow J$, it follows
from the well-known Varadhan's lemma that 
\[
\lim_{t\rightarrow\infty}\frac{1}{\beta(t)}\log\big(e^{-H(t)}\langle u(t,o)\rangle\big)=\sup\big\{-J-{\cal S}_{{\rm eu}}\big\}=-\inf\{J+{\cal S}_{{\rm eu}}\}.
\]
However, it requires a substantial amount of non-trivial effort to
implement this idea precisely even in the Euclidean case. The main
issue is that one cannot establish such an LDP on the whole (non-compact)
space $\mathbb{R}^{d}$; the Donsker-Varadhan LDP essentially only
holds when the state space is compact. As a result, one has to develop
localisation techniques carefully to reduce the problem to the ones
over fixed bounded domains. This is easier when one tries to prove
a lower asymptotics; by the positivity of the PAM one just directly
restrict the solution $u$ to a ball of fixed radius in order to get
a lower bound. The upper bound is much harder to obtain. To accturately
approximate the solution $u(t,o)$, one has to localise it on a sufficiently
large ball $Q_{R(t)}$ with a time-dependent radius $R(t)$ ($R(t)\rightarrow\infty$
as $t\rightarrow\infty$). One then decompose $Q_{R(t)}$ into balls
of fixed (time-independent) radius and controls the PAM over $R(t)$
in terms the corresponding ones over these subdomains. This step is
done through estimating the corresponding principal Dirichlet eigenvalues
(with a small sacrifice of potential), which is accurate enough to
reflect the large time asymptotics of the renormalised PAMs (Laplace's
principle). 

In the current hyperbolic situation, one faces several additional
difficulties. First of all, one cannot directly apply the Donsker-Varadhan
LDP since the underlying process $\{W_{s}^{t}:s\geqslant0\}$ depends
also on the $t$-parameter (the geometry changes at the same time
when one sends the time variable to infinity). To overcome this difficulty,
we shall rely on the abstract G\"artner-Ellis LDP in topological
vector spaces (cf. Section \ref{subsec:LDP} below). Justifying the
conditions in the theorem (i.e. convergence of cumulant generating
functions) becomes a question about spectral properties which can
be handled by analytical means (Lemma \ref{lem:ConvLogMGF}). In addition,
the decomposition of the large geodesic ball $Q_{R(t)}$ into congruent
balls is harder to achieve explicitly because one does not have the
obvious periodic lattice structure as in Euclidean spaces. We shall
develop a more robust way of constructing suitable coverings of $Q_{R(t)}$
as well as the associated partitions of unity (Sections \ref{subsec:HypPart}
and \ref{subsec:POU}). One then also needs to control the sacrifice
of potential in the eigenvalue estimates geometrically (Lemma
\ref{lem:POU} (ii)). A surprising fact is that even under exponential
volume growth in the hyperbolic space, the error arising from the
eigenvalue decomposition is still negligible with a proper choice
of $R(t)$ (Section \ref{subsec:DecomEigen}).

\subsection{Feynman-Kac representations}

Our proof of Theorem \ref{thm:MainThm} relies on the Feynman-Kac
representation as a starting point. We now recall this basic tool.

\subsubsection{The global and localised PAM}
\begin{prop}
\label{prop:FK}(i) {[}Global PAM{]} For a.e. realisation of $\xi$,
the solution $u(t,x)$ to the PAM defined by (\ref{eq:GlobalPAM})
admits the following Feynman-Kac representation:
\begin{equation}
u(t,x)=\mathbb{E}\big[e^{\int_{0}^{t}\xi(W_{s}^{x})ds}\big],\label{eq:FKGlobal}
\end{equation}
where $\{W_{s}^{x}\}$ denotes a hyperbolic Brownian motion (generated
by $\Delta$) starting at $x$ that is independent of $\xi$, and
the expectation is taking with respect to $\{W_{s}^{x}\}.$

\vspace{2mm}\noindent (ii) {[}Localised PAM{]} Recall that $Q_{R}$
is the closed geodesic ball of radius $R$ centered at $o$. Consider
the following initial-boundary value problem:
\begin{equation}
\begin{cases}
\partial_{s}u_{R}(s,x)=\Delta u_{R}(s,x)+\xi(x)u_{R}(s,x), & (s,x)\in(0,\infty)\times Q_{R};\\
u_{R}(0,\cdot)\equiv1;\ u_{R}(s,\cdot)=0\ \text{on }\partial Q_{R}
\end{cases}\label{eq:LocalPAM}
\end{equation}
Then solution $u_{R}(t,x)$ admits the following representation:
\[
u_{R}(s,x)=\mathbb{E}\big[e^{\int_{0}^{s}\xi(W_{r}^{x})dr};W^{x}([0,s])\subseteq Q_{R}\big].
\]
\end{prop}
\begin{proof}
This is standard from stochastic calculus; both equations (\ref{eq:GlobalPAM})
and (\ref{eq:LocalPAM}) are solved deterministically for every fixed
realisation of $\xi$ (cf. \cite{KS88}).
\end{proof}

\subsubsection{The rescaled equation}

We also need to consider a suitably rescaled PAM. Before doing so,
we first introduce a basic geometric convention that will be used
frequently in our discussion. Recall that the exponential map $\exp_{o}:T_{o}M\rightarrow M$
is a global diffeomorphism which produces the (global) geodesic polar
chart 
\[
[0,\infty)\times ST_{o}M\ni(\rho,\sigma)\mapsto\exp(\rho\sigma)\in M,
\]
and under this chart the metric tensor is given by (\ref{eq:RMetric}).
Therefore, one can forget about the actual space $M$ and instead
regard the space 
\begin{equation}
\Sigma\triangleq\{(\rho,\sigma):\rho\geqslant0,\sigma\in\mathbb{S}^{d-1}\}\label{eq:SigSp}
\end{equation}
equipped with metric $g=d\rho^{2}+\sinh^{2}\rho d\sigma^{2}$ as a
model for the space $M$. Here $d\sigma^{2}$ denotes the standard
metric on the unit sphere $\mathbb{S}^{d-1}$ in $\mathbb{R}^{d}$.
Of course when $\rho=0$ the sphere is collapsed to a point; topologically
$\Sigma\cong T_{o}M$ and geometrically $\rho=0$ is not a singularity. Later on, we will need to consider
spaces of curvature $\kappa_{t}\equiv-\alpha(t)^{2}$ for different
$t$'s and a corresponding family of Markov processes (the Brownian
motion) on these spaces. A benefit of this viewpoint is that these
spaces can all be modelled on the same (topological) space $\Sigma$
defined by (\ref{eq:SigSp}) but now equipped with the $t$-dependent
metric 
\begin{equation}
g^{t}=d\rho^{2}+\alpha(t)^{-2}\sinh^{2}(\alpha(t)\rho)d\sigma^{2}.\label{eq:MetricT}
\end{equation}
The space $(\Sigma,g^{t})$ is isometric to $\mathbb{H}_{\alpha(t)}^{d}$.
Note that since $\alpha(t)\rightarrow0$ as $t\rightarrow\infty,$
the space $(\Sigma,g^{\infty})$ is just the Euclidean space. Respectively,
given $R\geqslant0$ we introduce the ball 
\begin{equation}
\Sigma_{R}\triangleq\{(\rho,\sigma):\rho\leqslant R,\mathbb{S}^{d-1}\}.\label{eq:SigR}
\end{equation}
The spaces $\{(\Sigma_{R},g^{t}):t>0\}$ are topologically identical
but geometrically non-isometric; for each fixed $t$ it is the closed
geodesic ball of radius $R$ under curvature $\kappa_{t}\equiv-\alpha(t)^{2}$.
The space $(\Sigma_{R},g^{\infty})$ is the closed Euclidean $R$-ball.
Under such a convention (and under geodesic polar coordinates), the
rescaled Gaussian field $\xi_{t}$ (cf. (\ref{eq:RescaledXi})) and
the limiting functional $J$ (cf. Lemma \ref{lem:JtConv}) can both
be viewed as defined on $\Sigma$ or $\Sigma_{R}$ depending on the
context.

\vspace{2mm}Now we derive the equation for a suitably rescaled PAM
and its Feynman-Kac representation. Given fixed $R,t>0$, let $u_{R\alpha(t)}(\cdot,\cdot)$
be the solution to (\ref{eq:LocalPAM}) on $Q_{R\alpha(t)}$. Define
the function
\begin{equation}
v^{t}(s,x)\triangleq e^{-s\alpha^{2}(t)H(t)/t}u_{R\alpha(t)}(\alpha^{2}(t)s,\alpha(t)\cdot x),\ \ \ s\geqslant0,x\in Q_{R}.\label{eq:RescaleLocalPAM}
\end{equation}
We shall use polar coordinates $x=(\rho,\sigma)$ ($\rho\leqslant R$
and $\sigma\in ST_{o}M\cong\mathbb{S}^{d-1}$) so that $v^{t}(s,\cdot)$
is viewed as a function defined on $\Sigma_{R}$ (cf. (\ref{eq:SigR})). 
\begin{lem}
\label{lem:ResPAM}The function $\{v^{t}(s,x):s\geqslant0,x=(\rho,\sigma)\in\Sigma_{R}\}$
is the solution to the following rescaled PAM:
\begin{equation}
\begin{cases}
\partial_{s}v^{t}={\cal L}^{t}v^{t}+\xi_{t}v^{t}, & (s,(\rho,\sigma))\in(0,\infty)\times\Sigma_{R};\\
v^{t}(s,\cdot)=0 & \text{on }\{\rho=R\};\\
v^{t}(0,\cdot)\equiv1,
\end{cases}\label{eq:ResPAM}
\end{equation}
where
\[
{\cal L}^{t}\triangleq\partial_{\rho}^{2}+\alpha(t)\coth(\alpha(t)\rho)\partial_{\rho}+\alpha^{2}(t)\sinh^{-2}(\alpha(t)\rho)\Delta_{\sigma}
\]
is the Laplacian under curvature $\kappa_{t}\equiv-\alpha^{2}(t)$
with $(\rho,\sigma)$ being viewed as the corresponding geodesic polar
coordinates. In particular, $v^{t}(s,x)$ admits the following Feynman-Kac
representation:
\begin{equation}
v^{t}(s,x)=\mathbb{E}\big[e^{\int_{0}^{s}\xi_{t}(W_{r}^{t,x})dr};W^{t,x}([0,s])\subseteq\Sigma_{R}\big],\ \ \ (s,x)\in[0,\infty)\times\Sigma_{R},\label{eq:FKRes}
\end{equation}
where $\{W_{s}^{t,x}\}$ is a Brownian motion on $(\Sigma,g^{t})$
starting at $x$ that is independent of $\xi.$
\end{lem}
\begin{proof}
By the definition of $v^{t}(s,x),$ one has 
\begin{align*}
\frac{\partial v^{t}}{\partial s}(s,x)= & -\frac{\alpha^{2}(t)H(t)}{t}v^{t}(s,x)+e^{-s\alpha^{2}(t)H(t)/t}\alpha^{2}(t)\partial_{s}u_{R\alpha(t)}\big(\alpha^{2}(t)s,\alpha(t)\cdot x\big)\\
= & -\frac{\alpha^{2}(t)H(t)}{t}v^{t}(s,x)+e^{-s\alpha^{2}(t)H(t)/t}\alpha^{2}(t)\Delta u_{R\alpha(t)}\big(\alpha^{2}(t)s,\alpha(t)\cdot x\big)\\
 & \ \ \ +\alpha^{2}(t)\xi\big(\alpha(t)\cdot x\big)v^{t}(s,x)\\
= & \xi_{t}(x)v^{t}(s,x)+e^{-s\alpha^{2}(t)H(t)/t}\alpha^{2}(t)\Delta u_{R\alpha(t)}\big(\alpha^{2}(t)s,\alpha(t)\cdot x\big).
\end{align*}
By using the expression of $\Delta$ under geodesic polar chart (cf.
(\ref{eq:DeltaPolar})), with the substitution $r=\alpha(t)\rho$
one finds that 
\begin{align*}
 & \alpha^{2}(t)\Delta u_{R\alpha(t)}\big(\alpha^{2}(t)s,(\alpha(t)\rho,\sigma)\big)\\
 & =\alpha^{2}(t)\big[(\partial_{r}^{2}u_{R\alpha(t)}\big(\alpha^{2}(t)s,(\alpha(t)\rho,\sigma)\big)+\coth(\alpha(t)\rho)\partial_{r}u_{R\alpha(t)}\big(\alpha^{2}(t)s,(\alpha(t)\rho,\sigma)\big)\\
 & \ \ \ +\sinh^{-2}(\alpha(t)\rho)\partial_{\sigma}^{2}u_{R\alpha(t)}\big(\alpha^{2}(t)s,(\alpha(t)\rho,\sigma)\big)\big]\\
 & =\partial_{\rho}^{2}\big(u_{R\alpha(t)}(\alpha^{2}(t)s,(\alpha(t)\rho,\sigma))\big)+\alpha(t)\partial_{\rho}\big(u_{R\alpha(t)}(\alpha^{2}(t)s,(\alpha(t)\rho,\sigma))\big)\\
 & \ \ \ +\alpha^{2}(t)\sinh^{-2}(\alpha(t)\rho)\partial_{\sigma}^{2}\big(u_{R\alpha(t)}(\alpha^{2}(t)s,(\alpha(t)\rho,\sigma))\big)\\
 & ={\cal L}^{t}\big(u_{R\alpha(t)}(\alpha^{2}(t)s,(\alpha(t)\rho,\sigma))\big).
\end{align*}
The result thus follows from the observation that ${\cal L}^{t}$
is precisely the Laplacian under curvature $\kappa_{t}\equiv-\alpha^{2}(t)$
with $(\rho,\sigma)$ being viewed as the corresponding geodesic polar
coordinates (cf. (\ref{eq:LapAlp})).
\end{proof}
\noindent \textit{Notation}. Throughout the rest, we will always use
$\mathbb{E}[\cdot]$ to denote the expectation with respect to the
Brownian motion in the Feynman-Kac representation and use $\langle\cdot\rangle$
to denote the expectation with respect to $\xi$.

\subsection{The first order asymptotics}

Before developing the actual proof of Theorem \ref{thm:MainThm},
it is beneficial to point out that the first order asymptotics can
be obtained in a quite straight forward and robust way. 
\begin{prop}
\label{prop:1stOrdAsym}Let $\xi$ be a mean zero Gaussian field on
a complete Riemannian manifold $M$ with continuous covariance function
$C(x,y)\triangleq\mathbb{E}[\xi(x)\xi(y)]$ and constant variance
$C(x,x)\equiv\sigma^{2}$. Let $u(t,x)$ be the solution to the global
PAM defined by (\ref{eq:GlobalPAM}). Then for any positive integer
$p$$,$ one has 
\[
\lim_{t\rightarrow\infty}\frac{1}{H(pt)}\log\langle u(t,x)^{p}\rangle=1
\]
for any $x\in M$, where $H(t)\triangleq\sigma^{2}t^{2}/2$. 
\end{prop}
\begin{proof}
This is essentially the same as the proof of \cite[Theorem 4.1]{CM95}
and we reproduce the argument for the sake of completeness.
First of all, according to the Feynman-Kac representation (\ref{eq:FKGlobal}),
one can write 
\[
\langle u(t,x)^{p}\rangle=\mathbb{E}_{x}\big[\exp\big(\frac{1}{2}\sum_{i,j=1}^{p}\int_{0}^{t}\int_{0}^{t}C(W_{u}^{i},W_{v}^{j})dudv\big)\big].
\]
Here $W^{1},\cdots,W^{d}$ are independent Brownian motions (Markov
processes generated by $\Delta$) on $M$ all starting at $x$. Since
$|C(x,y)|\leqslant\sigma^{2},$ one immediately obtains that 
\[
\langle u(t,x)^{p}\rangle\leqslant e^{H(pt)}\ \text{or }\frac{1}{H(pt)}\log\langle u(t,x)^{p}\rangle\leqslant1.
\]

For the other direction, given $\varepsilon>0$ let $\delta$ be such
that 
\[
y,z\in B(x,\delta)\triangleq\{y\in M:d(y,x)<\delta\}\implies C(y,z)\geqslant\sigma^{2}-\varepsilon,
\]
This is possible due to the continuity of $B$. One then has the following
localised estimate:
\begin{align}
\langle u(t,x)^{p}\rangle & \geqslant\mathbb{E}_{x}\big[\exp\big(\frac{1}{2}\sum_{i,j=1}^{p}\int_{0}^{t}\int_{0}^{t}C(W_{u}^{i},W_{v}^{j})dudv\big);W^{i}([0,t])\subseteq B(x,\delta)\ \forall i\big]\nonumber \\
 & \geqslant\exp\big(\frac{1}{2}p^{2}(\sigma^{2}-\varepsilon)t^{2}\big)\times\mathbb{P}\big(W^{1}([0,t])\subseteq B(x,\delta)\big)^{p}.\label{eq:Lower1OrdPf}
\end{align}
On the other hand, it is standard that 
\begin{equation}
\mathbb{P}\big(W^{1}([0,t])\subseteq B(x,\delta)\big)\sim Ke^{-\lambda_{\delta}t}\ \ \ \text{as }t\rightarrow\infty,\label{eq:ExitTDecay}
\end{equation}
where $K$ is a positive constant and $\lambda_{\delta}\in(0,\infty)$
is the principal Dirichlet eigenvalue of $\Delta$ on $B(x,\delta)$.
After substituting (\ref{eq:ExitTDecay}) into (\ref{eq:Lower1OrdPf}),
one finds that 
\[
\underset{t\rightarrow\infty}{\underline{\lim}}\frac{1}{H(pt)}\log\langle u(t,x)^{p}\rangle\geqslant1-\frac{\varepsilon}{\sigma^{2}}.
\]
The desired lower asymptotics follows by taking $\varepsilon\downarrow0$.
\end{proof}
The main effort of the present article is thus to establish the fluctuation
asymptotics (\ref{eq:MainAsym}). In what follows, we develop the
main ingredients its the proof precisely. We begin by establishing
the lower $L^{1}$ asymptotics (i.e. $p=1$) in Section \ref{sec:Lower}
and then partially using it to establish the upper $L^{1}$ asymptotics
in Section \ref{sec:Upper}. After that, we identify the fluctuation
exponent in Section \ref{subsec:FlucExp}. Finally, we extend the
argument to the $p$-th moment asymptotics in Section \ref{sec:LpAsym}.

\section{\label{sec:Lower}The lower $L^{1}$ asymptotics}

The proof of Theorem \ref{thm:MainThm} contains two separate parts:
the lower and upper asymptotics. We will first derive the lower asymptotics
and then derive the upper asymptotics (whose proof will make use of
the lower bound at some point). For both asymptotics, we will focus
on the $L^{1}$ case. The general $p$-th moment case is discussed
in Section \ref{sec:LpAsym}.

Our first step is to establish the following result. Recall that $u$
is the solution to the PAM (\ref{eq:GlobalPAM}) under Gaussian potential
$\xi$. For each $R>0,$ we define the exponent
\begin{equation}
\chi_{R}\triangleq\inf\Big\{ J(\mu)+\sup_{f\in C_{b}(\Sigma_{R})}\{\int_{\Sigma_{R}}fd\mu+\lambda_{0}^{{\rm eu};f,R}\}:\mu\in{\cal P}(\Sigma_{R})\Big\}.\label{eq:ChiR}
\end{equation}

\begin{thm}
\label{thm:LowerBd}One has 
\begin{equation}
\underset{t\rightarrow\infty}{\underline{\lim}}\frac{1}{\beta(t)}\log\big(e^{-H(t)}\langle u(t,0)\rangle\big)\geqslant\underset{R\rightarrow\infty}{\overline{\lim}}(-\chi_{R}).\label{eq:LowerBd}
\end{equation}
\end{thm}
In the following subsections, we develop the main ingredients for
the proof of Theorem \ref{thm:LowerBd}.

\subsection{An eigenvalue asymptotics lemma}

We first establish a lemma about asymptotics of principal eigenvalues
that will be useful for proving the theorem. Recall that $(\Sigma_{R},g^{t})$
($0<t\leqslant\infty$) is the closed geodesic $R$-ball under curvature
$\kappa_{t}\equiv-\alpha(t)^{2}$.
\begin{lem}
\label{lem:VfAsymp}Let $f\in C_{b}(\Sigma_{R})$ (space of bounded
continuous functions on $\Sigma_{R}$) be given fixed. Let $v(s,(\rho,\sigma))$
be the solution to the following initial-boundary value problem:
\begin{equation}
\begin{cases}
\partial_{s}v={\cal L}^{t}v+fv, & (s,(\rho,\sigma))\in(0,\infty)\times\Sigma_{R};\\
v(s,\cdot)=0 & \text{on }\{\rho=R\};\\
v(0,\cdot)\equiv1.
\end{cases}\label{eq:ResPAMF}
\end{equation}
Then one has 
\[
\lim_{t\rightarrow\infty}\frac{1}{\beta(t)}\log v(\beta(t),(\rho,\sigma))=-\lambda_{0}^{{\rm eu}}
\]
uniformly for $(\rho,\sigma)\in[0,R-\varepsilon]\times\mathbb{S}^{d-1}$
for any given fixed $\varepsilon\in(0,R)$. Here $\lambda_{0}^{{\rm eu}}$
denotes the principal Dirichlet eigenvalue of $\Delta_{{\rm eu}}^{+}-f$
on $(\Sigma_{R},g^{\infty})$ with $\Delta_{{\rm eu}}^{+}$ being
the positive Euclidean Laplacian.
\end{lem}
\begin{proof}
Let $\{(\lambda_{n}^{t},\phi_{n}^{t}):n\geqslant0\}$ be the spectral
decomposition of $-({\cal L}^{t}+f)$ on $\Sigma_{R}$ with Dirichlet
boundary condition. Here ${\cal L}^{t}$ is the $g^{t}$-Laplacian
and the underlying measure on $\Sigma_{R}$ is the Riemannian volume
measure $d^{t}x$ induced by $g^{t}$. More specifically, 
\[
\lambda_{0}^{t}<\lambda_{1}^{t}\leqslant\lambda_{2}^{t}\leqslant\cdots\uparrow\infty
\]
are the Dirichlet eigenvalues of $-({\cal L}^{t}+f)$ on $\Sigma_{R}$
and $\{\phi_{n}^{t}:n\geqslant0\}$ is the corresponding ONB of $L^{2}(\Sigma_{R},d^{t}x)$
consisting of Dirichlet eigenfunctions. It is well-known from standard
spectral theory that $\lambda_{0}^{t}>0$ and $\phi_{0}^{t}$ is strictly
positive in $\Sigma_{R}.$ Under the spectral decomposition, one can
write 
\[
v(\beta(t),x)=\sum_{n=0}^{\infty}e^{-\lambda_{n}^{t}\cdot\beta(t)}\langle\phi_{n}^{t},{\bf 1}\rangle_{t} \, \phi_{n}^{t}(x),
\]
where $x=(\rho,\sigma)$ and $\langle\cdot,\cdot\rangle_{t}$ is the
$L^{2}$-inner product with respect to the volume measure $d^{t}x.$
It follows that 
\begin{equation}
\frac{1}{\beta(t)}\log v(\beta(t),x)=-\lambda_{0}^{t}+\frac{1}{\beta(t)}\log\big(\eta(t,x)+A(t,x)\big),\label{eq:DecomLogv}
\end{equation}
where 
\[
\eta(t,x)\triangleq\langle\phi_{0}^{t},{\bf 1}\rangle_{t}\phi_{0}^{t}(x),\ A(t,x)\triangleq\sum_{n=1}^{\infty}e^{-(\lambda_{n}^{t}-\lambda_{0}^{t})\beta(t)}\langle\phi_{n}^{t},{\bf 1}\rangle_{t} \, \phi_{n}^{t}(x).
\]
\textcolor{black}{Since $g^{t}$ converges to the Euclidean metric
$g^{{\rm eu}}$ as $t\rightarrow\infty$ on the compact space $\Sigma_{R}$,
the following facts are standard from perturbation theory (cf. \cite{RS78}).}

\vspace{2mm}\noindent \textcolor{black}{(i) }${\color{black}\underset{t\rightarrow\infty}{\lim}\lambda_{0}^{t}=\lambda_{0}^{{\rm eu}}}$.

\vspace{2mm}\noindent \textcolor{black}{(ii) There exist constants
}${\color{black}0<\eta_{1}<\eta_{2}}$ \textcolor{black}{such that }
\begin{equation}
{\color{black}\eta_{1}\leqslant\eta(t,(\rho,\sigma))\leqslant\eta_{2}\ \ \ \text{for all large }t\ \text{and}\ \rho\leqslant R-\varepsilon.}\label{eq:PertP2}
\end{equation}
\textcolor{black}{(iii) There exists $\delta>0$ such that 
\begin{equation}
\lambda_{1}^{t}-\lambda_{0}^{t}>\delta\ \ \ \text{for all large \ensuremath{t}}.\label{eq:PertP3}
\end{equation}
(iv) Let $H^{t}(s,x,y)$ be the Dirichlet heat kernel for ${\cal L}^{t}+f$
on $\Sigma_{R}.$ Then one has 
\begin{equation}
\sup_{t\geqslant1,x\in\Sigma_{R}}\big|e^{\lambda_{0}^{t}}\big(H^{t}(1,x,x)-e^{-\lambda_{0}^{t}}\phi_{0}^{t}(x)^{2}\big)\big|<\infty.\label{eq:PertP4}
\end{equation}
}

\vspace{2mm} We now proceed to estimate $A(t,x).$ Firstly, by using
the Cauchy-Schwarz inequality one has 
\begin{align*}
|A(t,x)| & \leqslant\sqrt{\sum_{n=1}^{\infty}\langle\phi_{n}^{t},{\bf 1}\rangle_{t}^{2}}\cdot\sqrt{\sum_{n=1}^{\infty}e^{-2(\lambda_{n}^{t}-\lambda_{0}^{t})\beta(t)}\phi_{n}^{t}(x)^{2}}\\
 & \leqslant\|{\bf 1}\|_{t}\cdot\sqrt{\sum_{n=1}^{\infty}e^{-2(\lambda_{n}^{t}-\lambda_{0}^{t})\beta(t)}\phi_{n}^{t}(x)^{2}}.
\end{align*}
In addition, one also has
\begin{align*}
\sum_{n=1}^{\infty}e^{-2(\lambda_{n}^{t}-\lambda_{0}^{t})\beta(t)}\phi_{n}^{t}(x)^{2} & \leqslant e^{-(\lambda_{1}^{t}-\lambda_{0}^{t})\beta(t)}\cdot\sum_{n=1}^{\infty}e^{-(\lambda_{n}^{t}-\lambda_{0}^{t})}\phi_{n}^{t}(x)^{2}\\
 & =e^{-(\lambda_{1}^{t}-\lambda_{0}^{t})\beta(t)}\cdot e^{\lambda_{0}^{t}}\big(H^{t}(1,x,x)-e^{-\lambda_{0}^{t}}\phi_{0}^{t}(x)^{2}\big).
\end{align*}
According to the properties (\ref{eq:PertP3}) and (\ref{eq:PertP4}),
one finds that 
\begin{equation}
|A(t,x)|\leqslant Me^{-\delta\beta(t)}\ \text{for all large \ensuremath{t},}\label{eq:EigenAEst}
\end{equation}
where $M,\delta$ are some positive constants independent of $t,x$.
It follows from (\ref{eq:EigenAEst}) and (\ref{eq:PertP2}) that
\[
\log\big(\eta_{1}-Me^{-\delta\beta(t)}\big)\leqslant\log\big(\eta(t,x)+A(t,x)\big)\leqslant\log(\eta_{2}+Me^{-\delta\beta(t)})
\]
uniformly for all large $t$ and $\rho\leqslant R-\varepsilon.$ Now
the desired result follows from the above Property (i) as well as
the relation (\ref{eq:DecomLogv}).
\end{proof}

\subsection{\label{subsec:LDP}A large deviation principle for localised hyperbolic
Brownian motions}

The core ingredient for proving Theorem \ref{thm:LowerBd} is a suitable
LDP which we describe as follows. 

We first make some preparations. Let $E\triangleq{\cal P}(\Sigma_{R})$
(respectively, $X\triangleq{\cal M}(\Sigma_{R})$) denote the space
of probability measures (respectively, finite signed measures) over
$\Sigma_{R}.$ The topology on $X$ is chosen to be the one generated
by 
\begin{equation}
\big\{\beta\in X:\big|\int_{\Sigma}fd(\beta-\alpha)\big|<r\big\}\ \ \ (f\in C_{b}(\Sigma_{R}),\alpha\in X).\label{eq:Top}
\end{equation}
The following facts are standard and can be found in \cite[Section 3.2]{DS89}.
\begin{lem}
\label{lem:Riesz}(i) The topology $\tau$ restricted on $E$ is exactly
the topology of weak convergence.

\vspace{2mm}\noindent (ii) The representation 
\begin{equation}
f\mapsto\big[\alpha\mapsto\int_{\Sigma}fd\alpha\big]\label{eq:Riesz}
\end{equation}
is a linear isomorphism between $C_{b}(\Sigma)$ and $X^{*}$.
\end{lem}
Note that $E$ is compact (since $\Sigma_{R}$ is compact). Given
$\mu\in{\cal P}(E)$, its \textit{logarithmic moment generating function}
(\textit{log m.g.f.}) is the functional $\Lambda_{\mu}:X^{*}\rightarrow(-\infty,\infty]$
defined by 
\[
\Lambda_{\mu}(\lambda)\triangleq\log\int_{E}\exp\big(_{X^{*}}\langle\lambda,x\rangle_{X}\mu(dx)\big),\ \ \ \lambda\in X^{*}.
\]
Under the correspondence (\ref{eq:Riesz}), $\Lambda_{\mu}$ can be
viewed as a functional on $C_{b}(\Sigma_{R})$: 
\[
\Lambda_{\mu}(f)=\log\int_{E}\exp\Big(\int_{\Sigma_{R}}f(x)\alpha(dx)\Big)\mu(d\alpha),\ \ \ f\in C_{b}(\Sigma_{R}).
\]

Throughout the rest, let $\varepsilon>0$ be given fixed. Let $\{W_{s}^{t}:s\geqslant0\}$
be a $g^{t}$-Brownian motion on $\Sigma$ ($g^{t}$ is the metric
with curvature $\kappa_{t}\equiv-\alpha(t)^{2}$ defined in (\ref{eq:MetricT})).
Let $\{L_{s}^{t}(dx):s\geqslant0\}$ be the occupation time measure
process defined by (\ref{eq:OTMP}). Consider the family of probability
measures 
\[
\{\mu_{t}^{(\rho,\sigma)}:t\geqslant0,\ (\rho,\sigma)\in\Sigma_{R},\rho\leqslant R-\varepsilon\}\subseteq{\cal P}(E)
\]
given by 
\[
\mu_{t}^{(\rho,\sigma)}(\Gamma)\triangleq\mathbb{E}_{\rho,\sigma}\big[L_{\beta(t)}^{t}\in\Gamma\big|W^{t}([0,\beta(t)])\subseteq\Sigma_{R}\big],\ \ \ \Gamma\in{\cal B}(E).
\]
Here the subscript $(\rho,\sigma)$ means that $W_{0}^{t}=(\rho,\sigma).$ 

As we will see, a (uniform) LDP for the family $\{\mu_{t}^{(\rho,\sigma)}\}$
will follow from the general theory of LDP in topological vector spaces
(the abstract G\"artner-Ellis theorem). To establish such an LDP,
a crucial point is to identify the underlying rate function. Let $\Lambda_{t}^{(\rho,\sigma)}$
be the log m.g.f. of $\mu_{t}^{(\rho,\sigma)}.$ One checks by definition
that 
\[
\Lambda_{t}^{(\rho,\sigma)}(f)=\log\mathbb{E}_{\rho,\sigma}\Big[\exp\big(\frac{1}{\beta(t)}\int_{0}^{\beta(t)}f(W_{s}^{t})ds\big)\big|W^{t}([0,\beta(t)])\subseteq\Sigma_{R}\Big]
\]
for any $f\in C_{b}(\Sigma_{R})$. Given any such $f$, let $\lambda_{0}^{{\rm eu};f,R}$
denote the principal Dirichlet eigenvalue of $-(\Delta_{{\rm eu}}+f)$
on $\Sigma_{R}$. We simply write $\lambda_{0}^{{\rm eu};R}$ for
the case $f=0$.
\begin{lem}
\label{lem:ConvLogMGF}Let $f\in C_{b}(\Sigma_{R})$ be given fixed.
Then one has 
\begin{equation}
\lim_{t\rightarrow\infty}\frac{1}{\beta(t)}\Lambda_{t}^{(\rho,\sigma)}(\beta(t)f)=\lambda_{0}^{{\rm eu};R}-\lambda_{0}^{{\rm eu};f,R}=:\Lambda(f),\label{eq:Lambda}
\end{equation}
where the convergence holds uniformly for $\rho\leqslant R-\varepsilon$.
\end{lem}
\begin{proof}
Let $v^{t,f}(s,(\rho,\sigma))$ be the solution to the initial-boundary
value problem (\ref{eq:ResPAMF}). According to the Feynman-Kac representation
(\ref{eq:FKRes}), one has 
\[
\exp\big(\Lambda_{t}^{(\rho,\sigma)}(\beta(t)f)\big)=\frac{v^{t,f}(\beta(t),(\rho,\sigma))}{v^{t,0}(\beta(t),(\rho,\sigma))}.
\]
The result thus follows from Lemma \ref{lem:VfAsymp}.
\end{proof}
Let $\Lambda^{*}:X\rightarrow[0,\infty]$ be the \textit{Fenchel-Legendre
transform} of $\Lambda$:
\begin{align}
\Lambda^{*}(\alpha) & \triangleq\sup\big\{_{X^{*}}\langle\lambda,\alpha\rangle_{X}-\Lambda(\lambda):\lambda\in X^{*}\big\}\nonumber \\
 & =\sup\Big\{\int_{\Sigma_{R}}fd\alpha+\lambda_{0}^{{\rm eu};f,R}:f\in C_{b}(\Sigma_{R})\Big\}-\lambda_{0}^{{\rm eu};R},\ \ \ \alpha\in X.\label{eq:Lamb*}
\end{align}
This will be the rate function governing the LDP for $\{\mu_{t}^{(\rho,\sigma)}\}$. 
\begin{prop}
\label{prop:LocalLDP}The family $\{\mu_{t}^{(\rho,\sigma)}\}$ satisfies
the following (uniform) LDP with the convex, good rate function $\Lambda^{*}$:
\begin{align}
-\inf_{\mathring{\Gamma}}\Lambda^{*}\leqslant & \underset{t\rightarrow\infty}{\underline{\lim}}\frac{1}{\beta(t)}\log\big(\inf_{\rho\leqslant R-\varepsilon}\mu_{t}^{(\rho,\sigma)}(\Gamma)\big)\nonumber \\
 & \ \ \ \ \ \ \ \ \leqslant\underset{t\rightarrow\infty}{\overline{\lim}}\frac{1}{\beta(t)}\log\big(\sup_{\rho\leqslant R-\varepsilon}\mu_{t}^{(\rho,\sigma)}(\Gamma)\big)\leqslant-\inf_{\bar{\Gamma}}\Lambda^{*}\label{eq:LocalLDP}
\end{align}
for all $\Gamma\in{\cal B}(E)$.
\end{prop}
We are going to prove Proposition \ref{prop:LocalLDP} by using the
abstract \textit{G\"artner-Ellis theorem} in topological vector spaces. It
asserts that under suitable conditions, Varadhan's asymptotics for
linear functionals implies an LDP. The precise result is recalled
as follows (cf. \cite[Corollary 4.6.14]{DZ09}). A function $f:X^{*}\rightarrow\mathbb{R}$
is said to be \textit{G\^ateaux differentiable} if for every $\alpha,\beta\in X,$
the function $t\mapsto f(\alpha+t\beta)$ is differentiable at $t=0.$

\begin{abGEthm} Let $\{\mu_{\eta}\}$ be
an exponentially tight family of Borel probability measures on a locally
convex, Hausdorff topological vector space $X.$ Suppose that 
\begin{equation}
\Lambda(\cdot)\triangleq\lim_{\eta\rightarrow0}\eta\Lambda_{\mu_{\eta}}(\cdot/\eta):X^{*}\rightarrow\mathbb{R}\label{eq:ConvLogMGF}
\end{equation}
exists finitely and is G\^ateaux differentiable. Then $\{\mu_{\eta}\}$
satisfies the LDP with the convex, good rate function $\Lambda^{*}$
being the Fenchel-Legendre transform of $\Lambda$.
\end{abGEthm}

\begin{proof}[Proof of Proposition \ref{prop:LocalLDP}]
We first prove the G\^ateaux differentiability of $\Lambda(\cdot)$ defined by (\ref{eq:Lambda}).
Let $f,g\in C_{b}(\Sigma_{R})$
and $\delta>0$ be given. To ease notation, we simply write
\[
\lambda_{\delta}\triangleq\lambda_{0}^{{\rm eu};f+\delta g,R},\ \lambda\triangleq\lambda_{0}^{{\rm eu};f,R}.
\]
Define 
\[
{\cal L}^{\delta}\triangleq\Delta_{{\rm eu}}^{+}-(f+\delta g),\ {\cal L}\triangleq\Delta_{{\rm eu}}^{+}-f.
\]
Let $u_{\delta}$ (respectively, $u$) be the normalised principal
eigenfunction of ${\cal L}^{\delta}$ (respectively, ${\cal L}$)
on $\Sigma_{R}$ with Dirichlet boundary condition. In particular,
one has 
\[
{\cal L}^{\delta}u_{\delta}=\lambda_{\delta}u_{\delta},\ {\cal L}u=\lambda u.
\]
Simply algebra shows that 
\[
(\lambda_{\delta}-\lambda)u_{\delta}=-\delta gu_{\delta}+{\cal L}(u_{\delta}-u)-\lambda(u_{\delta}-u).
\]
Since ${\cal L}$ is self-adjoint, one has 
\[
\langle{\cal L}(u_{\delta}-u),u\rangle=\langle u_{\delta}-u,{\cal L}u\rangle=\lambda\langle u_{\delta}-u,u\rangle,
\]
where $\langle\cdot,\cdot\rangle$ is the $L^{2}$-inner product with
respect to the Lebesgue measure. It follows that 
\[
(\lambda_{\delta}-\lambda)\langle u_{\delta},u\rangle=-\delta\langle gu_{\delta},u\rangle.
\]
\textcolor{black}{From standard perturbation results (cf. \cite{RS78}), one knows that
$u_{\delta}\rightarrow u$ in $L^{2}(\Sigma_{R})$ as $\delta\rightarrow0$. }As a result, the function $\delta\mapsto\lambda_{\delta}$
is differentiable at $\delta=0$ and one has 
\[
\frac{d}{d\delta}\big|_{\delta=0}\lambda_{\delta}=-\langle gu,u\rangle=-\int_{\Sigma}g(x)u(x)^{2}dx.
\]
This proves the G\^ateaux differentiability of the limiting functional
$\Lambda.$

To prove the LDP estimate (\ref{eq:LocalLDP}), let $\Gamma \in \bB(E)$ be given fixed. For every $t>0$, there exist $(\rho^t, \sigma^t)$ and $(\rho_t, \sigma_t) \in \Sigma_{R-\eps}$ such that
\begin{equation*}
 \mu_t^{(\rho_t, \sigma_t)}(\Gamma) = \inf_{\rho \leqslant R-\eps} \mu_t^{(\rho,\sigma)}(\Gamma)\;, \qquad    \mu_t^{(\rho^t, \sigma^t)}(\Gamma) = \sup_{\rho \leqslant R-\eps} \mu_t^{(\rho,\sigma)}(\Gamma)\;.
\end{equation*}
The uniform convergence \eqref{eq:Lambda} ensures that the convergence still holds (with the same limit $\Lambda$ when $(\rho,\sigma)$ is replaced by the sequences $(\rho_t, \sigma_t)$ and $(\rho^t, \sigma^t)$). Note that the compactness of $E$ trivially implies the uniform exponential
tightness of the family $\{\mu_{t}^{(\rho,\sigma)}\}$. Together with the G\^ateaux differentiability of $\Lambda$ and the abstract G\"artner-Ellis theorem, one deduces that both $\mu_t^{(\rho_t,\sigma_t)}$ and $\mu_t^{(\rho^t,\sigma^t)}$ satisfy the LDP with good rate function $\Lambda^*$. By the specific choices of $(\rho_t,\sigma_t)$ and $(\rho^t,\sigma^t)$, this in particular implies that \eqref{eq:LocalLDP} holds for $\Gamma \in \bB(E)$. The proof is then complete since $\Gamma \in \bB(E)$ is arbitrary.

\end{proof}

\subsection{Varadhan's asymptotics: lower bound}

An important consequence of the LDP (\ref{eq:LocalLDP}) is the following
Varadhan-type asymptotics, which is critical for the proof of Theorem
\ref{thm:LowerBd}. Here we state the lower asymptotics that is relevant
to the theorem (the upper version will be discussed in Lemma \ref{lem:VaradUp}
below where we prove the upper asymptotics). 

Recall that $\xi_{t}$ is the rescaled Gaussian field defined by (\ref{eq:RescaledXi})
and $J_{t}$ is the functional given by (\ref{eq:JtDef}). Topologically,
they are defined on $\Sigma$ and ${\cal P}_{c}(\Sigma)$ respectively,
where $\Sigma\cong T_{o}M$ is topologically identified with $M$.
However, one should keep in mind that they are defined under the geometry
of curvature $\kappa_{t}\equiv-\alpha(t)^{2}$, i.e. on the Riemannian
manifold $(\Sigma,g^{t}).$ Also recall that $W^{t}$ is a Brownian
motion in $(\Sigma,g^{t})$.
\begin{lem}
\label{lem:VaradLow}Let $R,\varepsilon>0$ be given fixed. Then the
following lower estimate holds true:
\begin{align*}
\underset{t\rightarrow\infty}{\underline{\lim}}\frac{1}{\beta(t)}\inf_{\rho\leqslant R-\varepsilon}\log & \mathbb{E}_{\rho,\sigma}\big[e^{-\beta(t)J_{t}(L_{\beta(t)}^{t})};W^{t}([0,\beta(t)])\subseteq\Sigma_{R}\big]\geqslant-\chi_{R},
\end{align*}
where $J$ is the functional defined by Lemma \ref{lem:JtConv} and
the exponent $\chi_{R}$ is defined by (\ref{eq:ChiR}).
\end{lem}
\begin{proof}
Recall that $-\lambda_{0}^{{\rm eu};f,R}$ (respectively, $-\lambda_{0}^{{\rm eu};R}$)
is the principal Dirichlet eigenvalue of $\Delta_{{\rm eu}}^{+}-f$
on $\Sigma_{R}$ (respectively, $\Delta_{{\rm eu}}^{+}$). One can
write 
\begin{align}
 & \frac{1}{\beta(t)}\log\mathbb{E}_{\rho,\sigma}\big[e^{-\beta(t)J_{t}(L_{\beta(t)}^{t})};W^{t}([0,\beta(t)])\subseteq\Sigma_{R}\big]\nonumber \\
 & =\frac{1}{\beta(t)}\log\mathbb{E}_{\rho,\sigma}\big[e^{-\beta(t)J_{t}(L_{\beta(t)}^{t})}\big|W^{t}([0,\beta(t)])\subseteq\Sigma_{R}\big]\nonumber \\
 & \ \ \ +\frac{1}{\beta(t)}\log\mathbb{P}_{\rho,\sigma}\big(W^{t}([0,\beta(t)])\subseteq\Sigma_{R}\big).\label{eq:LowerVarPf}
\end{align}
To study the asymptotics of the first term, given any $\mu\in{\cal P}(\Sigma_{R})$
and $r>0$ one has 
\begin{align*}
 & \mathbb{E}_{\rho,\sigma}\big[e^{-\beta(t)J_{t}(L_{\beta(t)}^{t})}\big|W^{t}([0,\beta(t)])\subseteq\Sigma_{R}\big]\\
 & =\int_{{\cal P}(\Sigma_{R})}e^{-\beta(t)J_{t}(\alpha)}d\mu_{t}^{(\rho,\sigma)}(\alpha)\\
 & \geqslant\exp\big(-\beta(t)\sup_{\alpha\in B(\mu,r)}J_{t}(\alpha)\big)\cdot\mu_{t}^{(\rho,\sigma)}\big(B(\mu,r)\big),
\end{align*}
where $B(\mu,r)\triangleq\{\nu\in{\cal P}(\Sigma_{R}):\rho(\nu,\mu)<r\}$
and $\rho$ is some metric compatible with the weak topology. It then
follows from the lower bound in (\ref{eq:LocalLDP}) that 
\begin{equation}
\underset{t\rightarrow\infty}{\underline{\lim}}\frac{1}{\beta(t)}\log\inf_{\rho\leqslant R-\varepsilon}{\rm L.H.S.}\geqslant-\sup_{B(\mu,r)}J-\inf_{B(\mu,r)}\Lambda^{*}.\label{eq:PfVAsym}
\end{equation}
Here we also used the fact that 
\[
\sup_{B(\mu,r)}J_{t}\rightarrow\sup_{B(\mu,r)}J\ \ \ \text{as }t\rightarrow\infty,
\]
which is a consequence of the uniform convergence of $J_{t}$ proved
in Lemma \ref{lem:JtConv}. By taking $r\downarrow0$ in (\ref{eq:PfVAsym}),
one obtains that 
\begin{align*}
 & \underset{t\rightarrow\infty}{\underline{\lim}}\frac{1}{\beta(t)}\log\inf_{\rho\leqslant R-\varepsilon}\mathbb{E}_{\rho,\sigma}\big[e^{-\beta(t)J_{t}(L_{\beta(t)}^{t})}\big|W^{t}([0,\beta(t)])\subseteq\Sigma_{R}\big]\\
 & \geqslant-\big(J(\mu)+\Lambda^{*}(\mu)\big)=-\Big(J(\mu)+\sup_{f\in C_{b}(\Sigma_{R})}\Big\{\int_{\Sigma_{R}}fd\mu+\lambda_{0}^{{\rm eu};f,R}\Big\}-\lambda_{0}^{{\rm eu};R}\Big).
\end{align*}
On the other hand, the second term in (\ref{eq:LowerVarPf}) converges
to $-\lambda_{0}^{{\rm eu};R}$ uniformly for $(\rho,\sigma)\in\Sigma_{R-\varepsilon}$
due to Lemma \ref{lem:ConvLogMGF}. The result thus follows after
cancelling out $\lambda_{0}^{{\rm eu};R}$ and taking infimum over
$\mu\in{\cal P}(\Sigma_{R})$.
\end{proof}

\subsection{Proof of Theorem \ref{thm:LowerBd}}

Gathering the previous ingredients together, we are now in a position
to prove Theorem \ref{thm:LowerBd}.

Let $R>0$ be given fixed. First of all, due to the stationarity of
$\xi$ one has 
\begin{align}
\langle u(t,o)\rangle & =\frac{1}{|Q_{R\alpha(t)}|}\int_{Q_{R\alpha(t)}}\langle u(t,x)\rangle dx\geqslant\frac{1}{|Q_{R\alpha(t)}|}\int_{Q_{R\alpha(t)}}\langle u_{R\alpha(t)}(t,x)\rangle dx\nonumber \\
 & =\frac{(d-1)2^{d-1}}{\omega_{d-1}e^{(d-1)R\alpha(t)}}\int_{0}^{R\alpha(t)}\int_{\mathbb{S}^{d-1}}\langle u_{R\alpha(t)}(t,(r,\sigma))\rangle\sinh^{d-1}rdrd\sigma,\label{eq:LowerBdA}
\end{align}
where $u_{R\alpha(t)}(\cdot,\cdot)$ is the solution to the localised
PAM (\ref{eq:LocalPAM}) on $Q_{R\alpha(t)}$ and we used the obvious
bound that 
\begin{align*}
|Q_{R\alpha(t)}| & =\omega_{d-1}\int_{0}^{R\alpha(t)}\sinh^{d-1}rdr\\
 & \leqslant\omega_{d-1}\int_{0}^{R\alpha(t)}\big(\frac{e^{r}}{2}\big)^{d-1}dr\leqslant\frac{\omega_{d-1}}{(d-1)2^{d-1}}e^{(d-1)R\alpha(t)}.
\end{align*}
Respectively, recall that the rescaled function 
\[
v^{t}(\beta(t),(\rho,\sigma))=e^{-H(t)}u_{R\alpha(t)}(t,(\alpha(t)\rho,\sigma)),\ \ \ (\rho,\sigma)\in\Sigma_{R}
\]
is the solution to the rescaled PAM (\ref{eq:ResPAM}). By applying
a change of variables $r=\alpha(t)\rho$ to (\ref{eq:LowerBdA}) and
using the Feynman-Kac representation (\ref{eq:FKRes}), one finds
that 
\begin{align}
 & e^{-H(t)}\langle u(t,o)\rangle\nonumber \\
 & \geqslant\frac{(d-1)2^{d-1}\alpha(t)}{\omega_{d-1}e^{(d-1)R\alpha(t)}}\int_{0}^{R}\int_{\mathbb{S}^{d-1}}\mathbb{E}_{\rho,\sigma}\big[\langle e^{\int_{0}^{\beta(t)}\xi_{t}(W_{s}^{t})ds}\rangle;W^{t}([0,\beta(t)])\subseteq\Sigma_{R}\big]\nonumber \\
 & \ \ \ \times\sinh^{d-1}(\alpha(t)\rho)d\rho d\sigma\nonumber \\
 & =\frac{(d-1)2^{d-1}\alpha(t)}{\omega_{d-1}e^{(d-1)R\alpha(t)}}\int_{0}^{R}\int_{\mathbb{S}^{d-1}}\mathbb{E}_{\rho,\sigma}\big[e^{-\beta(t)J_{t}(L_{\beta(t)}^{t})};W^{t}([0,\beta(t)])\subseteq\Sigma_{R}\big]\nonumber \\
 & \ \ \ \times\sinh^{d-1}(\alpha(t)\rho)d\rho d\sigma.\label{eq:LowerBdB}
\end{align}

Now let $\varepsilon,\eta>0$ be given. According to Lemma \ref{lem:VaradLow},
one has
\[
\inf_{\rho\leqslant R-\varepsilon}\mathbb{E}_{\rho,\sigma}\big[e^{-\beta(t)J_{t}(L_{\beta(t)}^{t})};W^{t}([0,\beta(t)])\subseteq\Sigma_{R}\big]\geqslant e^{\beta(t)(-\chi_{R}-\eta)}
\]
for all sufficiently large $t$, where $\chi_{R}$ is the exponent
defined by (\ref{eq:ChiR}). It follows from (\ref{eq:LowerBdB})
that
\begin{align*}
e^{-H(t)}\langle u(t,o)\rangle & \geqslant\frac{(d-1)2^{d-1}\alpha(t)}{\omega_{d-1}e^{(d-1)R\alpha(t)}}e^{\beta(t)(-\chi_{R}-\eta)}\int_{0}^{R-\varepsilon}\int_{\mathbb{S}^{d-1}}\sinh^{d-1}(\alpha(t)\rho)d\rho d\sigma\\
 & =\frac{(d-1)2^{d-1}\alpha(t)}{e^{(d-1)R\alpha(t)}}e^{\beta(t)(-\chi_{R}-\eta)}\int_{0}^{R-\varepsilon}\sinh^{d-1}(\alpha(t)\rho)d\rho.
\end{align*}
On the other hand, it is elementary to see that 
\[
\frac{(d-1)2^{d-1}\alpha(t)}{e^{(d-1)R\alpha(t)}}\int_{0}^{R-\varepsilon}\sinh^{d-1}(\alpha(t)\rho)d\rho\sim C_{R,\varepsilon,d}\alpha(t)^{d}\ \ \ \text{as }t\rightarrow\infty
\]
with some positive constant $C_{R,\varepsilon,d}.$ Therefore, one
arrives at 
\begin{equation}
\underset{t\rightarrow\infty}{\underline{\lim}}\frac{1}{\beta(t)}\log\big(e^{-H(t)}\langle u(t,o)\rangle\big)\geqslant-\chi_{R}-\eta.\label{eq:LowerBdC}
\end{equation}
The desired estimate (\ref{eq:LowerBd}) follows by taking $\eta\downarrow0$
and $R\uparrow\infty$.

\section{\label{sec:Upper}The upper $L^{1}$ asymptotics}

Our second step is to establish the following upper asymptotics.
\begin{thm}
\label{thm:UpperBd}One has 
\begin{equation}
\underset{t\rightarrow\infty}{\overline{\lim}}\frac{1}{\beta(t)}\log\big(e^{-H(t)}\langle u(t,0)\rangle\big)\leqslant\underset{R\rightarrow\infty}{\underline{\lim}}(-\chi_{R}),\label{eq:UpperBd}
\end{equation}
where $\chi_{R}$ is defined by (\ref{eq:ChiR}).
\end{thm}
In the following subsections, we develop the main ingredients for
the proof of Theorem \ref{thm:UpperBd}.

\subsection{An exit time estimate for hyperbolic Brownian motion}

We first present a lemma that will be useful for our localisation
argument later on. Let $\tau_{R}$ denote the first exit time of the
hyperbolic Brownian motion $\{W_{t}^{o}\}$ on $M$ (starting at $o$)
for the geodesic ball $Q_{R}$ (here curvature $\kappa\equiv-1$). 
\begin{lem}
\label{lem:ExitEst}There exist constants $C_{1},C_{2}>0,$ such that
\begin{equation}
\mathbb{P}(\tau_{R}\leqslant t)\leqslant C_{1}e^{-C_{2}R^{2}/t}\label{eq:ExitEst}
\end{equation}
for all $R\gg t>1$.
\end{lem}
\begin{proof}
Let $\rho_{t}$ be the radial component of $W_{t}^{o}$. From the
expression (\ref{eq:DeltaPolar}) of $\Delta$ in geodesic polar coordinates,
it is readily seen that $\rho_{t}$ is a Markov process with generator
$\partial_{\rho}^{2}+\coth\rho\partial_{\rho}.$ In addition, it is
obvious that 
\[
\mathbb{P}(\tau_{R}\leqslant t)\leqslant\mathbb{P}(\tau_{R}'\leqslant t),
\]
where $\tau_{R}'\triangleq\inf\{s\geqslant0:\rho_{s}^{1}=R\}$ and
$\rho_{t}^{1}$ is the radial process starting at $1$. The process
$\rho_{t}^{1}$ satisfies the following one-dimensional SDE:
\[
d\rho_{t}=\sqrt{2}dB_{t}+\coth\rho_{t}dt,\ \rho_{0}=1.
\]

We now fix a smooth, non-decreasing function $f:[0,\infty)\rightarrow[0,\infty)$
such that $f(x)=1/2$ on $[0,1/2]$ and $f(x)=x$ on $[1,\infty)$.
By It\^o's formula, one has 
\[
f(\rho_{t}^{1})=1+\sqrt{2}\int_{0}^{t}f'(\rho_{s}^{1})dB_{s}+\int_{0}^{t}\big(f'(\rho_{s}^{1})\coth\rho_{s}^{1}+f''(\rho_{s}^{1})\big)ds.
\]
From the choice of $f$, it is clear that 
\[
\sup_{0\leqslant s\leqslant t}\rho_{s}^{1}\geqslant R\iff\sup_{0\leqslant s\leqslant t}f(\rho_{s}^{1})\geqslant R.
\]
In addition, one has 
\begin{align*}
\sup_{0\leqslant s\leqslant t}f(\rho_{s}^{1}) & \leqslant1+\sqrt{2}\sup_{0\leqslant s\leqslant t}\big|\int_{0}^{s}f'(\rho_{u}^{1})dB_{u}\big|+\int_{0}^{t}\big|f'(\rho_{s}^{1})\coth\rho_{s}^{1}+f''(\rho_{s}^{1})\big|ds\\
 & \leqslant C_{1}t+\sqrt{2}\sup_{0\leqslant s\leqslant t}\big|\int_{0}^{s}f'(\rho_{u}^{1})dB_{u}\big|.
\end{align*}
Since $R\gg t,$ it follows that 
\[
\mathbb{P}(\tau_{R}^{'}\leqslant t)=\mathbb{P}\big(\sup_{0\leqslant s\leqslant t}f(\rho_{s}^{1})\geqslant R\big)\leqslant\mathbb{P}\big(\sup_{0\leqslant s\leqslant t}\big|\int_{0}^{s}f'(\rho_{u}^{1})dB_{u}\big|\geqslant C_{2}R\big).
\]
The desired inequality (\ref{eq:ExitEst}) follows from the classical
martingale inequality 
\[
\mathbb{P}\big(\sup_{0\leqslant s\leqslant t}|M_{s}|\geqslant x,\langle M\rangle_{t}\leqslant y\big)\leqslant2e^{-\frac{x^{2}}{2y}}\ \ \ \forall x,y,t>0
\]
with $M_{t}\triangleq\int_{0}^{t}f'(\rho_{s}^{1})dB_{s}$ and $y\triangleq\|f'\|_{\infty}^{2}t$.
\end{proof}

\subsection{\label{subsec:Localisation}Localisation}

The following lemma allows one to localise the problem on a sufficiently
large ball. Recall that $u_{R}(t,x)$ is the solution to the localised
PAM (\ref{eq:LocalPAM}).
\begin{lem}
\label{lem:UpperLocal}Let $R(t)\gg t^{5/4}$ as $t\rightarrow\infty$.
Then one has 
\[
\lim_{t\rightarrow\infty}\frac{\langle u_{R(t)}(t,o)\rangle}{\langle u(t,o)\rangle}=1.
\]
\end{lem}
\begin{proof}
By using the Feynman-Kac formulae (cf. Proposition \ref{prop:FK}),
one has 
\begin{align*}
\langle u(t,o)-u_{R(t)}(t,o)\rangle & =\big\langle\mathbb{E}_{o}\big[e^{\int_{0}^{t}\xi(W_{s})ds};\tau_{R(t)}\leqslant t\big]\big\rangle\\
 & \leqslant\Big\langle\mathbb{E}_{o}\Big[\frac{1}{t}\int_{0}^{t}e^{\xi(W_{s})}ds;\tau_{R(t)}\leqslant t\Big]\Big\rangle\ \ \ (\text{Jensen})\\
 & =e^{H(t)}\mathbb{P}_{o}(\tau_{R(t)}\leqslant t).
\end{align*}
Here $W$ is the hyperbolic Brownian motion in $M$ and $\tau_{R(t)}$
denotes its exit time for the ball $Q_{R(t)}$. It follows from Lemma
\ref{lem:ExitEst} that
\begin{equation}
\langle u(t,o)-u_{R(t)}(t,o)\rangle\leqslant C_{1}e^{H(t)-C_{2}R(t)^{2}/t}.\label{eq:Loc1}
\end{equation}
On the other hand, the lower bound (\ref{eq:LowerBd}) shows that
for any given fixed number
\[
\chi'>\underline{\chi}\triangleq\underset{R\rightarrow\infty}{\underline{\lim}}\chi_{R},
\]
one has 
\begin{equation}
\langle u(t,o)\rangle\geqslant e^{H(t)-\chi'\beta(t)}\ \ \ \text{for }t\ \text{sufficiently large.}\label{eq:Loc2}
\end{equation}
Combining (\ref{eq:Loc1}) and (\ref{eq:Loc2}), one finds that 
\begin{equation}
\frac{\langle u(t,o)-u_{R(t)}(t,o)\rangle}{\langle u(t,o)\rangle}\leqslant C_{1}e^{-(C_{2}R(t)^{2}/t-\chi'\beta(t))}.\label{eq:Loc3}
\end{equation}
Since $\beta(t)=t^{3/2}$ (cf. (\ref{eq:BetaScale})), it is now clear
that the right hand side of (\ref{eq:Loc3}) tends to zero as $t\rightarrow\infty$
provided that $R(t)\gg t^{5/4}$.
\end{proof}

\subsection{\label{subsec:FixDom}A fixed-domain upper bound}

As in \cite{GK00}, later on we are going to decompose a large ($t$-dependent)
geodesic ball into regions of a fixed volume (independent of $t$).
An essential ingredient in the argument is a suitable upper asymptotics
for the rescaled PAM (\ref{lem:ResPAM}) on a fixed domain. Let $r>0$
be given fixed. Let $\{\lambda_{k}^{\xi_{t}}(\Sigma_{r}):k\geqslant1\}$
denote the decreasing sequence of Dirichlet eigenvalues for the operator
${\cal L}^{t}+\xi_{t}$ on $\Sigma_{r}$. 
\begin{lem}
\label{lem:DomainUpper}One has 
\begin{equation}
\underset{t\rightarrow\infty}{\overline{\lim}}\frac{1}{\beta(t)}\log\Big\langle\sum_{k=1}^{\infty}e^{\beta(t)\lambda_{k}^{\xi_{t}}(\Sigma_{r})}\Big\rangle\leqslant-\chi_{r},\label{eq:DomainUpper}
\end{equation}
where $\chi_{r}$ is defined by (\ref{eq:ChiR}). 
\end{lem}
The rest of this subsection is devoted to the proof of Lemma \ref{lem:DomainUpper}.
We begin by writing 
\[
\sum_{k=1}^{\infty}e^{\beta(t)\lambda_{k}^{\xi_{t}}(\Sigma_{r})}=\int_{\Sigma_{r}}q^{\xi_{t}}(\beta(t),x,x)d^{t}x,
\]
where $q^{\xi_{t}}$ is the Dirichlet heat kernel for the self-adjoint
operator ${\cal L}^{t}+\xi_{t}$ on $\Sigma_{r}$ and $d^{t}x$ is
the volume measure induced by the metric $g^{t}$. By using a $g^{t}$-Brownian
motion $W^{t}$, one can write
\[
q^{\xi_{t}}(\beta(t),x,x)=\mathbb{E}_{x}\big[e^{\beta(t)(L_{\beta(t)}^{t},\xi_{t})}\delta_{x}(W_{\beta(t)}^{t});W^{t}([0,\beta(t)])\subseteq\Sigma_{r}\big],
\]
where we recall that $L_{\beta(t)}^{t}$ is the occupation time measure
of $W^{t}$ up to $\beta(t)$. After taking expectation $\langle\cdot\rangle$,
one has 
\begin{equation}
\Big\langle\sum_{k=1}^{\infty}e^{\beta(t)\lambda_{k}^{\xi_{t}}(\Sigma_{r})}\Big\rangle=\int_{\Sigma_{r}}\mathbb{E}_{x}\big[e^{-\beta(t)J_{t}(L_{\beta(t)}^{t})}\delta_{x}(W_{\beta(t)}^{t});W^{t}([0,\beta(t)])\subseteq\Sigma_{r}\big]d^{t}x.\label{eq:Up1}
\end{equation}

Now let $\delta>0$ be given fixed. The concavity and positivity of
$J_{t}$ easily implies that
\[
\beta(t)J_{t}(L_{\beta(t)}^{t})\geqslant(\beta(t)-\delta)J_{t}(L_{\beta(t)-\delta}^{t}).
\]
It then follows from (\ref{eq:Up1}) and the Markov property that
\begin{align}
\Big\langle\sum_{k=1}^{\infty}e^{\beta(t)\lambda_{k}^{\xi_{t}}(\Sigma_{r})}\Big\rangle\leqslant & \int_{\Sigma_{r}}\mathbb{E}_{x}\big[e^{-(\beta(t)-\delta)J_{t}(L_{\beta(t)-\delta}^{t})}\nonumber \\
 & \ \ \ \times p^{t}(\delta,W_{\beta(t)-\delta}^{t},x);W^{t}([0,\beta(t)-\delta])\subseteq\Sigma_{r}\big]d^{t}x,\label{eq:Up3}
\end{align}
where $p^{t}(s,y,x)$ is now the $g^{t}$-hyperbolic heat kernel (i.e.
the transition density for the hyperbolic Brownian motion $W^{t}$).

To estimate the right hand side, we need two basic ingredients: a
hyperbolic heat kernel estimate for $p$ and an upper Varadhan asymptotics
for $L^{t}$.

\subsubsection{A hyperbolic heat kernel estimate}

Let $p_{\kappa}(s,x,y)$ denote the heat kernel on the hyperboloid
$\mathbb{H}_{\sqrt{\kappa}}^{d}$ with curvature $-\kappa$ ($\kappa>0$).
Since $p_{\kappa}$ depends only on $s$ and the hyperbolic distance
between $x,y$, one can simply write $p_{\kappa}=p_{\kappa}(s,\rho)$.
A simple scaling argument shows that 
\begin{equation}
p_{\kappa}(s,\rho)=\kappa^{d/2}p_{1}(\kappa s,\sqrt{\kappa}\rho).\label{eq:HKScale}
\end{equation}
The following upper bound for $p_{\kappa}$ is a direct consequence
of (\ref{eq:HKScale}) and the corresponding estimate for $p_{1}$
proved by Davies-Mandouvalos \cite[Theorem 3.1]{DM98}.
\begin{lem}
\label{lem:HKUpperGen}There exists a universal constant $C_{d}$
depending only on the dimension, such that 
\begin{align*}
p_{\kappa}(s,\rho)\leqslant & C_{d}s^{-d/2}\exp\big(-\frac{(d-1)^{2}\kappa s}{4}-\frac{\rho^{2}}{4s}-\frac{(d-1)\sqrt{\kappa}\rho}{2}\big)\\
 & \ \ \ \times\big(1+\sqrt{\kappa}\rho+\kappa s\big)^{\frac{d-3}{2}}(1+\sqrt{\kappa}\rho)
\end{align*}
for all $s>0,\rho\geqslant0.$
\end{lem}
In our situation, $\kappa=-\alpha(t)^{2}$ for the heat kernel $p^{t}$.
The following result is thus a direct corollary of Lemma \ref{lem:HKUpperGen}.
\begin{cor}
With some universal constant $C=C_{d,r}>0,$ one has 
\begin{equation}
p^{t}(\delta,y,x)\leqslant C\delta^{-d/2}\ \ \ \forall t,\delta>0,\ y,x\in\Sigma_{r}.\label{eq:HKUp}
\end{equation}
\end{cor}

\subsubsection{\label{subsec:VaradUp}Varadhan's asymptotics: upper bound}

Let $\varepsilon>0$ be also given fixed and denote $E\triangleq{\cal P}(\Sigma_{r+\varepsilon})$.
Exactly the same argument as in Section \ref{subsec:LDP} shows that
the family of probability measures
\[
\bar{\mu}_{t}^{(\rho,\sigma)}(\Gamma)\triangleq\mathbb{E}_{(\rho,\sigma)}\big[L_{\beta(t)-\delta}^{t}\in\Gamma\big|W^{t}([0,\beta(t)-\delta])\subseteq\Sigma_{r+\varepsilon}\big],\ \ \ \Gamma\in{\cal B}(E)
\]
satisfy the same LDP with good rate function $\Lambda^{*}$ defined
by (\ref{eq:Lamb*}) (with $R=r+\varepsilon$) uniformly with respect
to $(\rho,\sigma)\in\Sigma_{r}.$ The following result is the upper
bound counterpart of Lemma \ref{lem:VaradLow}.
\begin{lem}
\label{lem:VaradUp}One has 
\begin{align}
\underset{t\rightarrow\infty}{\overline{\lim}}\frac{1}{\beta(t)}\sup_{\rho\leqslant r}\log & \mathbb{E}_{\rho,\sigma}\big[e^{-(\beta(t)-\delta)J_{t}(L_{\beta(t)-\delta}^{t})};W^{t}([0,\beta(t)-\delta])\subseteq\Sigma_{r+\varepsilon}\big]\leqslant-\chi_{r+\varepsilon}.\label{eq:VaradUp}
\end{align}
\end{lem}
\begin{proof}
Let $L,\eta>0$ be given. Define $K_{L}\triangleq\{\alpha\in E:\Lambda^{*}(\alpha)\leqslant L\}.$
Note that $K_{L}$ is a compact subset of $E.$ For each $\alpha\in K_{L}$,
there exists $\rho_{\alpha}>0$ such that 
\[
\sup_{\bar{B}(\alpha,\rho_{\alpha})}(-J)\leqslant(-J)(\alpha)+\eta,\ \inf_{\bar{B}(\alpha,\rho_{\alpha})}\Lambda^{*}\geqslant\Lambda^{*}(\alpha)-\eta.
\]
By the compactness of $K_{L}$, one can cover it by finitely many
such balls, say $\bar{B}(\alpha_{m},\rho_{m})$ ($i=1,\cdots,N$).
Setting $G\triangleq\cup_{m=1}^{N}\bar{B}(\alpha_{m},\rho_{m})$,
one has 
\begin{align*}
 & \mathbb{E}_{\rho,\sigma}\big[e^{-(\beta(t)-\delta)J_{t}(L_{\beta(t)-\delta}^{t})}\big|W^{t}([0,\beta(t)-\delta])\subseteq\Sigma_{r+\varepsilon}\big]\\
 & =\int_{E}e^{-(\beta(t)-\delta)J_{t}(\alpha)}\bar{\mu}_{t}^{(\rho,\sigma)}(d\alpha)=\big(\int_{G}+\int_{G^{c}}\big)e^{-(\beta(t)-\delta)J_{t}(\alpha)}\bar{\mu}_{t}^{(\rho,\sigma)}(d\alpha)\\
 & \leqslant\sum_{m=1}^{N}\int_{\bar{B}(\alpha_{m},\rho_{m})}e^{-(\beta(t)-\delta)J_{t}(\alpha)}\bar{\mu}_{t}^{(\rho,\sigma)}(d\alpha)+\bar{\mu}_{t}^{(\rho,\sigma)}(G^{c}).
\end{align*}
It follows that 
\begin{align*}
 & \frac{1}{\beta(t)}\log\sup_{\rho\leqslant r}{\rm L.H.S.}\\
 & \leqslant\frac{1}{\beta(t)}\log(N+1)+\max\big\{\frac{\beta(t)-\delta}{\beta(t)}\sup_{\bar{B}(\alpha_{m},\rho_{m})}(-J_{t})\\
 & \ \ \ +\frac{1}{\beta(t)}\log\sup_{\rho\leqslant r}\bar{\mu}_{t}^{(\rho,\sigma)}(\bar{B}(\alpha_{m},\rho_{m})):1\leqslant m\leqslant N\big\}\vee\big(\frac{1}{\beta(t)}\log\sup_{\rho\leqslant r}\bar{\mu}_{t}^{(\rho,\sigma)}(G^{c})\big).
\end{align*}

We now analyse the behaviour of the right hand side as $t\rightarrow\infty$.
It is seen from Lemma \ref{lem:JtConv} and the uniform LDP that 
\[
\frac{1}{\beta(t)}\log(N+1)\rightarrow0,\ \frac{\beta(t)-\delta}{\beta(t)}\sup_{\bar{B}(\alpha_{m},\rho_{m})}(-J_{t})\rightarrow\sup_{\bar{B}(\alpha_{m},\rho_{m})}(-J)\leqslant(-J)(\alpha_{m})+\eta,
\]
\[
\underset{t\rightarrow\infty}{\overline{\lim}}\frac{1}{\beta(t)}\log\sup_{\rho\leqslant r}\bar{\mu}_{t}^{(\rho,\sigma)}(\bar{B}(\alpha_{m},\rho_{m}))\leqslant-\inf_{\bar{B}(\alpha_{m},\rho_{m})}\Lambda^{*}\leqslant-\Lambda^{*}(\alpha_{m})+\eta,
\]
\[
\underset{t\rightarrow\infty}{\overline{\lim}}\frac{1}{\beta(t)}\log\sup_{\rho\leqslant r}\bar{\mu}_{t}^{(\rho,\sigma)}(G^{c})\leqslant-\inf_{G^{c}}\Lambda^{*}\leqslant-L\ \ \ (\text{since }K_{L}\subseteq G).
\]
As a consequence, one obtains that
\begin{align*}
 & \underset{t\rightarrow\infty}{\overline{\lim}}\frac{1}{\beta(t)}\log\sup_{\rho\leqslant r}\mathbb{E}_{\rho,\sigma}\big[e^{-(\beta(t)-\delta)J_{t}(L_{\beta(t)-\delta}^{t})}\big|W^{t}([0,\beta(t)-\delta])\subseteq\Sigma_{r+\varepsilon}\big]\\
 & \leqslant\max\big\{2\eta+(-J)(\alpha_{m})-\Lambda^{*}(\alpha_{m}):1\leqslant m\leqslant N\big\}\vee(-L)\\
 & \leqslant\big(2\eta+\sup_{\mu\in E}(-J(\mu)-\Lambda^{*}(\mu))\big)\vee(-L).
\end{align*}
Since $\eta$ and $L$ are arbitrary, it follows that 
\begin{align*}
 & \underset{t\rightarrow\infty}{\overline{\lim}}\frac{1}{\beta(t)}\log\sup_{\rho\leqslant r}\mathbb{E}_{\rho,\sigma}\big[e^{-(\beta(t)-\delta)J_{t}(L_{\beta(t)-\delta}^{t})}\big|W^{t}([0,\beta(t)-\delta])\subseteq\Sigma_{r+\varepsilon}\big]\\
 & \leqslant-\inf_{\mu\in E}\big(J(\mu)+\Lambda^{*}(\mu)\big).
\end{align*}
The desired estimate now follows easily from (\ref{eq:Lamb*}) as
well as Lemma \ref{lem:VaradLow}.
\end{proof}

\subsubsection{Completing the proof of Lemma \ref{lem:DomainUpper}}

We now proceed to upper bound the right hand side of (\ref{eq:Up3})
(call that expression $I_{t}$). Let $\varepsilon>0$ be given as
before. By using the heat kernel bound (\ref{eq:HKUp}), one finds
that 
\begin{align*}
I_{t} & \leqslant C_{d,r}\delta^{-d/2}\int_{\Sigma_{r}}\mathbb{E}_{x}\big[e^{-(\beta(t)-\delta)J_{t}(L_{\beta(t)-\delta}^{t})};W^{t}([0,\beta(t)-\delta])\subseteq\Sigma_{r+\varepsilon}\big]d^{t}x\\
 & \leqslant C_{d,r}\delta^{-d/2}{\rm vol}^{t}(\Sigma_{r})\sup_{\rho\leqslant r}\mathbb{E}_{x}\big[e^{-(\beta(t)-\delta)J_{t}(L_{\beta(t)-\delta}^{t})};W^{t}([0,\beta(t)-\delta])\subseteq\Sigma_{r+\varepsilon}\big],
\end{align*}
where ${\rm vol}^{t}$ denotes the volume measure with respect to
the metric $g^{t}$. Note that ${\rm vol}^{t}(\Sigma_{r})$ converges
to the Euclidean volume of the $r$-ball as $t\rightarrow\infty.$
Therefore, one concludes from Lemma \ref{lem:VaradUp} that 
\[
\underset{t\rightarrow\infty}{\overline{\lim}}\frac{1}{\beta(t)}\log I_{t}\leqslant-\chi_{r+\varepsilon}.
\]
The desired estimate (\ref{eq:DomainUpper}) follows by taking $\varepsilon\searrow0.$

\subsection{\label{subsec:HypPart} An explicit decomposition of large geodesic
balls}

After localisation on a (large $t$-dependent) geodesic ball as in
Section \ref{subsec:Localisation}, an important step in the argument
is to decompose the ball into subdomains of fixed ($t$-independent)
volume. After that, one can apply the estimate developed in Section
\ref{subsec:FixDom} to each of these subdomains. 

To develop this step precisely, let $r<\tilde{R}(t)$ be given numbers.
Here we consider $r$ fixed and $\tilde{R}(t)\uparrow\infty$ as $t\rightarrow\infty$
(we will properly define $\tilde{R}(t)$ later on but its value is
of no importance at this stage). Let $Q_{\tilde{R}(t)}^{t}$ denote
the closed geodesic ball centered at $o$ with radius $\tilde{R}(t)$
on the hyperboloid model $\mathbb{H}_{\alpha(t)}^{d}$. We shall decompose
$Q_{\tilde{R}(t)}^{t}$ into subdomains whose volumes are all equal
to ${\rm vol}^{t}(Q_{r}^{t})$. The underlying idea is very simple:
one first partitions $Q_{\tilde{R}(t)}^{t}$ into concentric annuli
with hyperbolic width $r$ and on each annulus one further decomposes
the angular component (the unit sphere $\mathbb{S}^{d-1}$) into congruent
balls. The radius of these balls are determined by the volume matching
condition. 

\subsubsection{\label{subsec:SpherePack}Sphere packing and covering}

We first spend some time discussing the decomposition / economic covering
of the angular component $\mathbb{S}^{d-1}$. If $d=2$, this is straight
forward; one just evenly partition the unit circle into arcs of fixed
length. Some technical care is needed in higher dimensions. For our
purpose, one does not really need a strict partition of $\mathbb{S}^{d-1}$
(which is geometrically cumbersome to describe). Instead, an economic
covering by identical balls would suffice. Here we develop some basic
lemmas that will be needed for our later discussion. We use $B(\sigma,\theta)$
to denote the open ball of radius $\theta$ centered at $\sigma$
on the sphere $\mathbb{S}^{d-1}$. We also denote $\omega_{d-1}\triangleq{\rm vol}(\mathbb{S}^{d-1})$.

Let $\theta\in(0,\pi)$ be a given fixed number. Let ${\cal C}(\theta)=\{B(\sigma_{i},\theta/2):i\in{\cal I}\}$
denote a \textit{maximal sphere packing} of $\mathbb{S}^{d-1}$. In
other words, this is a maximal way of fitting disjoint balls of radius
$\theta/2$ into $\mathbb{S}^{d-1}$. Clearly, such a packing always
exists but needs not be unique. We fix such a choice. The total number
$N(\theta)$ of members in ${\cal C}(\theta)$ admits the following
obvious bound:
\begin{equation}
N(\theta)\leqslant\frac{\omega_{d-1}}{{\rm vol}(B(\sigma,\theta/2))}\leqslant\frac{C_{d}}{\theta^{d-1}}\ \ \ \forall\theta\in(0,\pi)\label{eq:NumSigBall}
\end{equation}
with some universal constant $C_{d}$ depending only on $d$. 
\begin{lem}
\label{lem:ThetaCov}The family ${\cal U}(\theta)\triangleq\{B(\sigma_{i},\theta):i\in{\cal I}\}$
is a cover of $\mathbb{S}^{d-1}.$ 
\end{lem}
\begin{proof}
Suppose on the contrary that there exists some $\sigma\in\mathbb{S}^{d-1}$
not being covered by the balls $B(\sigma_{i},\theta)$. Then the ball
$B(\sigma,\theta/2)$ does not intersect $B(\sigma_{i},\theta/2)$
for all $i$ and can thus be added into the family ${\cal C}(\theta)$
to form a larger packing of $\mathbb{S}^{d-1}$. This contradicts
the maximality of ${\cal C}(\theta)$.
\end{proof}
\begin{lem}
\label{lem:2ThetaCov}Consider the cover ${\cal V}(\theta)\triangleq\{B(\sigma_{i},2\theta):i\in{\cal I}\}$.
Then there exists a universal constant $D_{d}$ depending only on
$d$, such that any point $\sigma\in\mathbb{S}^{d-1}$ is covered
by at most $D_{d}$ times. 
\end{lem}
\begin{proof}
Let $\sigma\in\mathbb{S}^{d-1}$ be given fixed. Define ${\cal I}_{\sigma}\triangleq\{i\in{\cal I}:\sigma\in B(\sigma_{i},2\theta)\}$.
This is equivalent to saying that $\sigma_{i}\in B(\sigma,2\theta)$
for all $i\in{\cal I}_{\sigma}$. It follows that 
\[
\bigcup_{i\in{\cal I}_{\sigma}}B\big(\sigma_{i},\frac{\theta}{2}\big)\subseteq B\big(\sigma,\frac{5\theta}{2}\big).
\]
Since this is a disjoint union, one has 
\[
|{\cal I}_{\sigma}|\leqslant\frac{{\rm vol}\big(B(\sigma,5\theta/2)\big)}{{\rm vol}\big(B(\sigma_{i},\theta/2)\big)}\leqslant D_{d}.
\]
The result thus follows.
\end{proof}

\subsubsection{Construction of decomposition and related geometric estimates}

We now proceed to construct the decomposition of the geodesic ball
$Q_{\tilde{R}(t)}^{t}$ announced at the beginning. Under the hyperboloid
model $\mathbb{H}_{\alpha(t)}^{d}$, let $(\rho,\sigma)$ denote the
geodesic polar coordinates with respect to $o$. 

\vspace{2mm}\noindent (i) {[}Radial part{]} We first divide $Q_{\tilde{R}(t)}^{t}$
into concentric annuli with hyperbolic width $r$ (the inner one is
taken to be the $r$-ball centered at $o$). In other words, 
\[
Q_{\tilde{R}(t)}^{t}=\bigcup_{k=1}^{\tilde{R}(t)/r}A_{k}^{t},\ \ \ A_{k}^{t}\triangleq\{x=(\rho,\sigma)\in\mathbb{H}_{\alpha(t)}^{d}:(k-1)r\leqslant\rho\leqslant kr\}.
\]
(ii) {[}Angular part{]} For each $k,$ consider the cover 
\begin{equation}
{\cal U}(\theta_{k}^{t})=\{B(\sigma_{i}^{k,t},\theta_{k}^{t}):i\in{\cal I}_{k}^{t}\}\label{eq:UTheCover}
\end{equation}
of $\mathbb{S}^{d-1}$ defined by Lemma \ref{lem:ThetaCov}, where
the angular radius $\theta_{k}^{t}$ is determined by Lemma \ref{lem:AngRad}
below (volume matching) and recall that ${\cal U}(\theta_{k}^{t})$
is induced from a maximal sphere packing ${\cal C}(\theta_{k}^{t})$
(cf. Section \ref{subsec:SpherePack}). For each $i\in{\cal I}_{k}^{t},$
we define 
\begin{equation}
P_{k}^{t}(i)\triangleq\{(\rho,\sigma):(k-1)r\leqslant\rho\leqslant kr,\sigma\in B(\sigma_{i}^{k,t},\theta_{k}^{t})\}.\label{eq:Pkt}
\end{equation}
One is then led to the following decomposition of $Q_{\tilde{R}(t)}^{t}$,
on which our later analysis is largely based. 
\begin{equation}
Q_{\tilde{R}(t)}^{t}=\bigcup_{k=1}^{\tilde{R}(t)/r}\bigcup_{i\in{\cal I}_{k}^{t}}P_{k}^{t}(i).\label{eq:QDecomp}
\end{equation}
\begin{figure}[H]  
\begin{center}   
\includegraphics[scale=0.35]{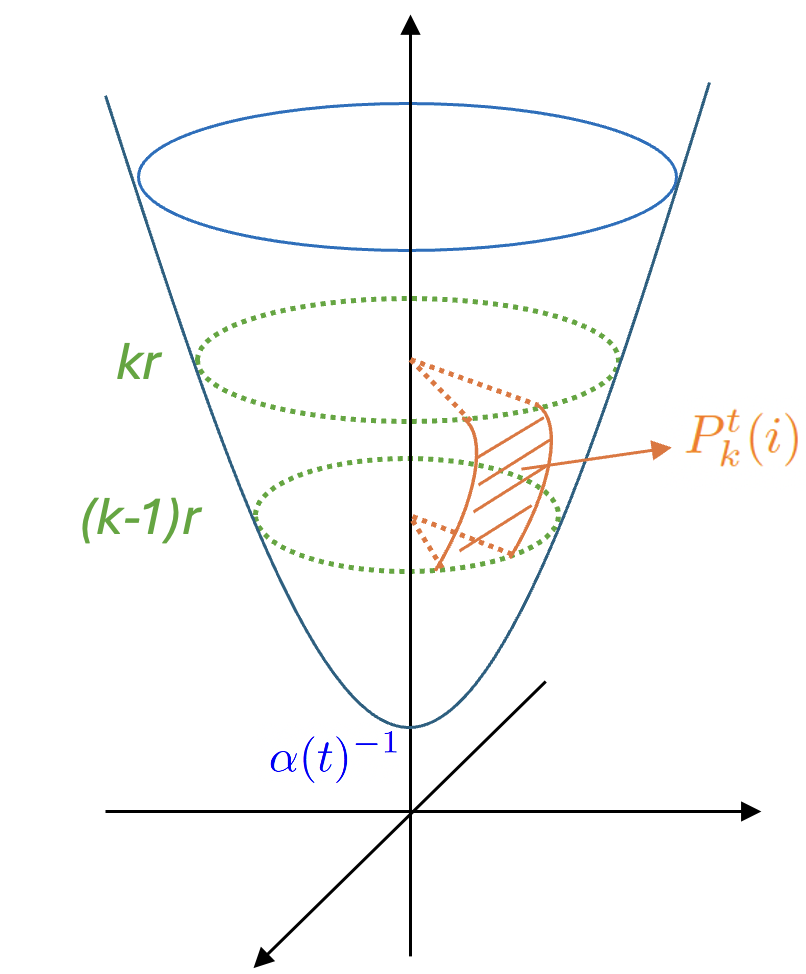}
\protect\caption{Decomposition of $Q^t_{\tilde{R}(t)}$ when $d=2$.}
\end{center} 
\end{figure}

The angular radius $\theta_{k}^{t}$ is determined by the volume matching
condition ${\rm vol}^{t}(P_{k}^{t}(i))={\rm vol}^{t}(Q_{r}^{t})$.
More precisely, one has the following lemma. We will require that
$\theta_{k}^{t}\in(0,\pi/2)$ (just take $\theta_{k}^{t}=\pi/2$ otherwise). 
\begin{lem}
\label{lem:AngRad}The number $\theta_{k}^{t}$ is determined by the
relation that 
\begin{equation}
\int_{0}^{\theta_{k}^{t}}\sin^{d-2}udu=\frac{\omega_{d-1}}{\omega_{d-2}}\times\frac{\int_{0}^{r}\sinh^{d-1}(\alpha(t)\rho)d\rho}{\int_{(k-1)r}^{kr}\sinh^{d-1}(\alpha(t)\rho)d\rho}.\label{eq:AngForm}
\end{equation}
It admits the following estimate: 
\begin{equation}
C_{1,d}\times\frac{\sinh(\alpha(t)r/2)}{\sinh(k\alpha(t)r)}\leqslant\theta_{k}^{t}\leqslant C_{2,d}\times\frac{\sinh(\alpha(t)r)}{\sinh((k-1)\alpha(t)r)}\ \ \ \forall k,t.\label{eq:EstAng}
\end{equation}
In addition, the total number $N_{r}^{t}$ of subdomains $P_{k}^{t}(i)$
are estimated as 
\begin{equation}
N_{r}^{t}\leqslant C_{3,d}\times\frac{\tilde{R}(t)}{r}\times\big(\frac{\sinh(\tilde{R}(t)\alpha(t))}{\sinh(\alpha(t)r/2)}\big)^{d-1}.\label{eq:NumEst}
\end{equation}
Here $C_{i,d}$ ($i=1,2,3$) are universal constants that depend only
on $d.$
\end{lem}
\begin{proof}
Recall that the volume of a ball of radius $\theta$ on $\mathbb{S}^{d-1}$
is given by 
\[
\omega_{d-2}\times\int_{0}^{\theta}\sin^{d-2}udu.
\]
According to the definition (\ref{eq:Pkt}) of $P_{k}^{t}(i)$ as
well as the expression (\ref{eq:VolAlp}) of the volume form on $\mathbb{H}_{\alpha(t)}^{d}$,
one has 
\begin{align*}
{\rm vol}^{t}(P_{k}^{t}(i)) & =\omega_{d-2}\times\int_{0}^{\theta_{k}^{t}}\sin^{d-2}udu\times\int_{(k-1)r}^{kr}\alpha(t)^{-(d-1)}\sinh^{d-1}(\alpha(t)\rho)d\rho.
\end{align*}
In the same way, one also has 
\begin{align*}
{\rm vol}^{t}(Q_{r}^{t}) & =\omega_{d-1}\times\int_{0}^{r}\alpha(t)^{-(d-1)}\sinh^{d-1}(\alpha(t)\rho)d\rho
\end{align*}
The relation (\ref{eq:AngForm}) follows by equating the above two
expressions. The estimate (\ref{eq:EstAng}) follows from the elementary
inequality 
\[
\frac{2u}{\pi}\leqslant\sin u\leqslant u\ \ \ \forall u\in\big[0,\frac{\pi}{2}\big]
\]
as well as the monotonicity of $x\mapsto\sinh x.$ The estimate (\ref{eq:NumEst})
follows from (\ref{eq:NumSigBall}) as well as (\ref{eq:EstAng}).
\end{proof}
The following lemma plays a crucial role in our analysis. It says
that each $P_{k}^{t}(i)$ can be fitted into some geodesic ball of
radius $\propto r$. We use ${\rm Diam}P$ to denote the diameter
of a subset $P$ in $\mathbb{H}_{\alpha(t)}^{d}$.
\begin{lem}
\label{lem:DiamEst}There exist a universal constant $C_{4,d}$ depending
only on $d$ and a constant $T_{r}$ depending additionally on $r$,
such that 
\begin{equation}
{\rm Diam}(P_{k}^{t}(i))\leqslant C_{4,d}r\label{eq:DiamEst}
\end{equation}
for all $t>T_{r}$ and all subdomains $P_{k}^{t}(i)$ in the decomposition
(\ref{eq:QDecomp}). In particular, each $P_{k}^{t}(i)$ is contained
in some geodesic ball of radius $C_{4,d}r$. 
\end{lem}
\begin{proof}
Let $A,B\in P_{k}^{t}(i)$. Denote $A',B'$ as the corresponding points
on the level set $\{x:d^{t}(x,o)=kr\}$ that have the same angular
component as $A,B$ respectively. It is clear that 
\begin{equation}
d^{t}(A,B)\leqslant2r+d^{t}(A',B').\label{eq:ABvsA'B'}
\end{equation}
By the distance formula (\ref{eq:HypDist}) on the hyperboloid $\mathbb{H}_{\alpha(t)}^{d}$,
one has 
\begin{equation}
A'*B'=-\alpha(t)^{-2}\cosh\big(\alpha(t)d^{t}(A',B')\big),\label{eq:A*B1}
\end{equation}
where $*$ denotes the standard Lorentzian inner product in $\mathbb{R}^{d+1}.$
Explicit calculation shows that 
\begin{equation}
A'*B'=\alpha(t)^{-2}\big(\sinh^{2}(k\alpha(t)r)\cos\angle_{{\rm \sigma}}A'B'-\cosh^{2}(k\alpha(t)r)\big),\label{eq:A*B2}
\end{equation}
where $\angle_{\sigma}A'B'$ denotes the angle between the angular
components $A'_{\sigma},B'_{\sigma}$ of $A',B'$ (both are unit vectors
on $\mathbb{S}^{d-1}$). Since $A'_{\sigma},B'_{\sigma}$ are contained
in the same ball of radius $\theta_{k}^{t}$ on $\mathbb{S}^{d-1}$,
one has $\angle_{\sigma}A'B'\leqslant2\theta_{k}^{t}$. By equating
(\ref{eq:A*B1}) and (\ref{eq:A*B2}), one finds that
\begin{align}
\cosh\big(\alpha(t)d^{t}(A',B')\big)-1 & =(1-\cos\angle_{{\rm \sigma}}A'B')\sinh^{2}(k\alpha(t)r)\nonumber \\
 & =\frac{1}{2}\sin^{2}\angle_{{\rm \sigma}}A'B'\sinh^{2}(k\alpha(t)r)\nonumber \\
 & \leqslant\frac{1}{2}(2\theta_{k}^{t})^{2}\sinh^{2}(k\alpha(t)r)\nonumber \\
 & \leqslant2C_{2,d}^{2}\times\sinh^{2}\alpha(t)r\times\big(\frac{\sinh(k\alpha(t)r)}{\sinh((k-1)\alpha(t)r)}\big)^{2},\label{eq:DiamEstPf}
\end{align}
where the last inequality follows from the estimate (\ref{eq:EstAng}).

We now estimate the above quotient. Denoting $x\triangleq\alpha(t)r$
and assuming $k\geqslant2$, it is elementary to see that 
\[
\frac{\sinh kx}{\sinh(k-1)x}=\frac{1}{e^{x}}+\frac{1-e^{-2x}}{e^{x}(e^{-2x}-e^{-2kx})}\leqslant\frac{1}{e^{x}}+\frac{1-e^{-2x}}{e^{x}(e^{-2x}-e^{-4x})}\leqslant3
\]
provided that $x\in(0,\delta)$ with some universal number $\delta.$
According to (\ref{eq:DiamEstPf}), one has 
\begin{equation}
\cosh\big(\alpha(t)d^{t}(A',B')\big)-1\leqslant18C_{2,d}^{2}\sinh^{2}\alpha(t)r.\label{eq:DistEst}
\end{equation}
Recall that $\alpha(t)=t^{-1/4}\downarrow0$ as $t\rightarrow\infty.$
One can choose $T_{r}$ so that 
\[
\alpha(t)r<\delta,\ \sinh\alpha(t)r\leqslant2\alpha(t)r,\ \ \ \forall t>T_{r}.
\]
It follows from (\ref{eq:DistEst}) that 
\[
\cosh\big(\alpha(t)d^{t}(A',B')\big)-1\leqslant72C_{2,d}^{2}\big(\alpha(t)r\big)^{2}\leqslant\cosh\big(12C_{2,d}\alpha(t)r\big)-1,
\]
where we also used the obvious inequality that $\cosh x\geqslant1+x^{2}/2$
for $x>0.$ As a consequence, 
\[
d^{t}(A',B')\leqslant12C_{2,d}r\ \ \ \forall t>T_{r}.
\]
By substituting this back into (\ref{eq:ABvsA'B'}), one obtains the
desired estimate (\ref{eq:DiamEst}) with $C_{4,d}\triangleq12C_{2,d}+2$.
\end{proof}

\subsection{\label{subsec:ParEst}Decomposition of principal eigenvalue with
respect to partition of unity}

Before returning in our geometric context, we shall first recall a
general estimate on principal Dirichlet eigenvalues with respect to
a given partition of unity. This is standard but we carefully state
a form that suits our purpose. Let $N$ be an oriented Riemannian
manifold with volume measure $dx$. 
\begin{defn}
\label{def:POU}A countable family $\{\varphi_{m}:m\in\mathbb{N}\}$
of smooth functions on $N$ is called a \textit{partition of unity}
on $M$ if the following properties hold true:

\vspace{2mm}\noindent (i) $0\leqslant\varphi_{m}\leqslant1$;\\
(ii) $D_{m}\triangleq{\rm supp}\varphi_{m}$ is compact;\\
(iii) for each $x\in N$, there exists a neighbourhood $U$ of $x$
which intersects at most finitely many $D_{m}$'s;\\
(iv) $\sum_{m}\varphi_{m}^{2}=1$.
\end{defn}
In the usual definition there is no square in Condition (iv); we impose
the square for the convenience of our purpose. We will also assume
that $D_{m}$ has piecewise smooth boundary (this will always be true
in our situation). Let $V:N\rightarrow\mathbb{R}$ be a continuous
function and let $D$ be a bounded domain in $N$ with piecewise smooth
boundary. We use $\lambda^{V}(D)$ to denote the principal Dirichlet
eigenvalue of $\Delta+V$. It is standard from the Rayleigh-Ritz formula
that 
\begin{equation}
\lambda^{V}(D)=\sup_{\substack{\psi\in C_{c}^{\infty}(D)\\
\|\psi\|_{L^{2}(D)}=1
}
}\int_{D}\big(-|\nabla\psi|^{2}+V\psi^{2}\big)dx.\label{eq:Rayleigh}
\end{equation}
The following estimate controls the principle Dirichlet eigenvalue
on $D$ (with a sacrifice on the potential) in terms of $\lambda^{V}(D_{m})$
(cf. \cite[Section 3.1, Proposition 1]{GK00}).
\begin{lem}
\label{lem:DEigenEst}Define the function $\Phi\triangleq\sum_{m}|\nabla\varphi_{m}|^{2}$
($\Phi$ is well-defined due to Condition (iii)). Then one has 
\begin{equation}
\lambda^{V-\Phi}(D)\leqslant\sup_{m}\lambda^{V}(D_{m}).\label{eq:DEigenEst}
\end{equation}
\end{lem}
\begin{proof}
Let $\psi\in C_{c}^{\infty}(D)$ be a test function. Define $\psi_{m}\triangleq\psi\varphi_{m}\in C_{c}^{\infty}(D_{m}).$
By using the Condition (iv), an easy calculation shows that 
\[
\sum_{m}|\nabla\psi_{m}|^{2}=|\nabla\psi|^{2}+\Phi\psi^{2}.
\]
It follows that 
\begin{align*}
\int_{D}\big(-|\nabla\psi|^{2}+(V-\Phi)\psi^{2}\big) & =\int_{M}\big(-\sum_{m}|\nabla\psi_{m}|^{2}+V\psi^{2}\big)\\
 & =\sum_{m}\int_{M}\big(-|\nabla\psi_{m}|^{2}+V\psi_{m}^{2}\big).
\end{align*}
On the other hand, for each $m$ one has 
\[
\int_{M}\big(-|\nabla\psi_{m}|^{2}+V\psi_{m}^{2}\big)\leqslant\lambda^{V}(D_{m})\|\psi_{m}\|_{L^{2}(D_{m})}^{2}\leqslant\sup_{m'}\lambda^{V}(D_{m'})\cdot\|\psi_{m}\|_{L^{2}(D_{m})}^{2}.
\]
After summing over $m$, one finds that 
\[
\int_{D}\big(-|\nabla\psi|^{2}+(V-\Phi)\psi^{2}\big)\leqslant\sup_{m'}\lambda^{V}(D_{m'})\cdot\sum_{m}\|\psi_{m}\|_{L^{2}(D_{m})}^{2}=\sup_{m'}\lambda^{V}(D_{m'})\cdot\|\psi\|_{L^{2}}^{2}.
\]
This gives the desired estimate. 
\end{proof}
\begin{rem}
In general, the sacrifice of potential (i.e. $\Phi$) in the estimate
(\ref{eq:DEigenEst}) depends on the magnitude of the gradient of
the partition of unity. There is no guarantee that $\Phi$ can be
made arbitrarily small; in contrast to the principal Neumann eigenvalue,
it is not true that $\lambda^{V}(D)\leqslant\sup_{m}\lambda^{V}(D_{m})$
for any partition of $D$ into subdomains $D_{m}.$ 
\end{rem}

\subsection{\label{subsec:POU}Construction of a partition of unity}

A key step in the proof of Theorem \ref{thm:UpperBd} is to construct
a partition of unity (in the sense of Definition \ref{def:POU}) which
respects the decomposition (\ref{eq:QDecomp}) constructed in Section
\ref{subsec:HypPart} (so that Lemma \ref{lem:DiamEst} can be applied)
and at the same time the sacrifice of potential $\Phi$ is small in
a certain sense. The main technical lemma to achieve this is stated
below. We continue to use the same notation as in Section \ref{subsec:HypPart}. 
\begin{lem}
\label{lem:POU}For each $r,t>1$, one can construct a partition of
unity 
\[
\{\varphi_{k,i}:1\leqslant k\leqslant\tilde{R}(t)/r,\ i\in{\cal I}_{k}^{t}\}
\]
on $Q_{\tilde{R}(t)}^{t}$ which satisfies the following properties.
For any $r>1$, there exists $T_{r}>0$ such that for all $t>T_{r}$,
one has:

\vspace{2mm}\noindent (i) for each $k,i,$ the support $\hat{P}_{k}^{t}(i)\triangleq{\rm supp}\varphi_{k,i}$
is contained in a geodesic ball of radius $K_{d}r$ with some universal
constant $K_{d}$ depending only on $d$;

\vspace{2mm}\noindent (ii) $\Phi_{r}^{t}\triangleq\sum_{k,i}|\nabla^{t}\varphi_{k,i}|_{t}^{2}\leqslant L_{d}/r^{2}$
with some universal constant $L_{d}$ depending only on $d$, where
$\nabla^{t}$, $|\cdot|_{t}$ denote the Riemannian gradient and metric
on $\mathbb{H}_{\alpha(t)}^{d}$ respectively.
\end{lem}
\begin{proof}
Our construction of $\{\varphi_{k,i}\}$ is very natural: it has the
form 
\begin{equation}
\varphi_{k,i}(\rho,\sigma)=\eta_{k}(\rho)\zeta_{k,i}(\sigma),\label{eq:POU}
\end{equation}
where $(\rho,\sigma)$ are the geodesic polar coordinates on $Q_{\tilde{R}(t)}^{t}$
with respect to the origin. 

\vspace{2mm}\noindent\textit{\uline{Step one: construction of \mbox{$\eta_{k}(\rho)$}}}.

\vspace{2mm}Let
$\varepsilon$ be a fixed small number ($\varepsilon=1/3$ suffices).
We define $\eta_{k}(\rho)$ to be a smooth function such that the
following properties hold true:

\vspace{2mm}\noindent (i) $0\leqslant\eta_{k}(\rho)\leqslant1$;\\
(ii) $\eta_{k}(\rho)=1$ for $\rho\in[(k-1+\varepsilon)r,(k-\varepsilon)r]$;\\
(iii) $\eta_{k}(\rho)=0$ on $[(k-1-\varepsilon)r,(k+\varepsilon)r]^{c}$;\\
(iv) $\sum_{k}\eta_{k}^{2}(\rho)=1$ for all $\rho\in[0,\tilde{R}(t)]$.

\vspace{2mm}\noindent Clearly, such a family $\{\eta_{k}\}$ exists
and the construction is elementary. It is easily seen that one can
choose $\eta_{k}$ so that 
\begin{equation}
|\eta_{k}'(\rho)|\leqslant\frac{1}{2\varepsilon r}\ \ \ \forall k,\rho.\label{eq:EtaBound}
\end{equation}
The figure below is a simple illustration of the construction. \begin{figure}[H]  
\begin{center}   
\includegraphics[scale=0.27]{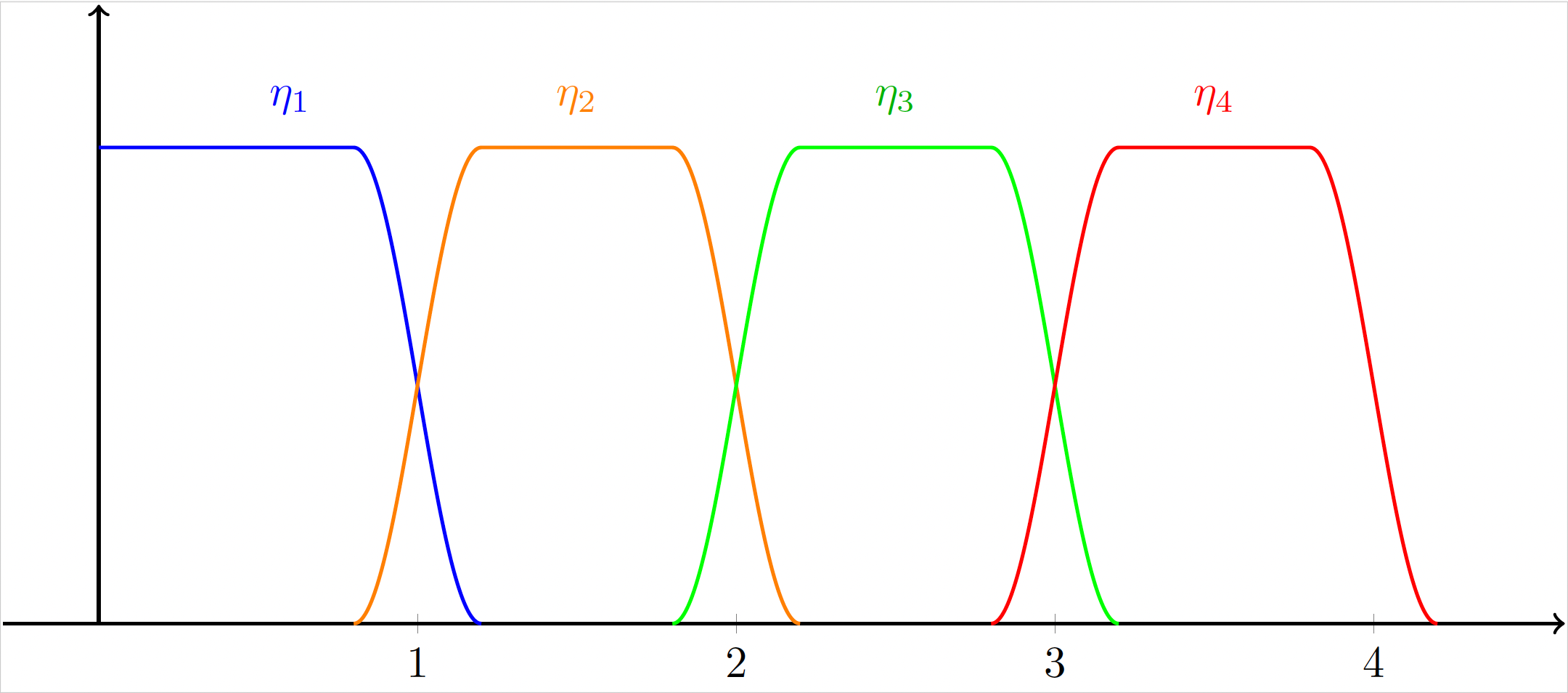}
\end{center} 
\end{figure}

\noindent\textit{\uline{Step two: construction of \mbox{$\zeta_{k,i}(\sigma)$}}}.

\vspace{2mm}This requires some care because one does not have an actual partition
of $\mathbb{S}^{d-1}$ here. Recall that for each fixed $k$ the sphere
$\mathbb{S}^{d-1}$ is covered by ${\cal U}(\theta_{k}^{t})$ (cf.
(\ref{eq:UTheCover})), where ${\cal C}(\theta_{k}^{t})=\{B(\sigma_{i},\theta_{k}^{t}/2):i\in{\cal I}_{k}^{t}\}$
is a maximal packing of $\mathbb{S}^{d-1}$. For each $i\in{\cal I}_{k}^{t}$,
let $\chi_{k,i}(\sigma)$ be a smooth function such that:

\vspace{2mm}\noindent (i) $0\leqslant\chi_{k,i}\leqslant1$;

\vspace{2mm}\noindent (ii) $\chi_{k,i}=1$ on $B(\sigma_{i},\theta_{k}^{t})$
and $\chi_{k,i}=0$ on $B(\sigma_{i},2\theta_{k}^{t})^{c}$;

\vspace{2mm}\noindent (iii) $|\nabla_{\sigma}\chi_{k,i}|\leqslant(\theta_{k}^{t})^{-1}$
where $\nabla_{\sigma}$ denotes the gradient on $\mathbb{S}^{d-1}$.

\vspace{2mm}\noindent The existence of $\chi_{k,i}$ is again elementary.
We now define 
\[
\zeta_{k,i}\triangleq\frac{\chi_{k,i}}{\sqrt{\sum_{j}\chi_{k,j}^{2}}},\ \ \ i\in{\cal I}_{k}^{t}.
\]
Note that $\zeta_{k,i}$ is well-defined; indeed $\sum_{j}\chi_{k,j}^{2}\geqslant1$
since ${\cal U}(\theta_{k}^{t})$ covers $\mathbb{S}^{d-1}$ (cf.
Lemma \ref{lem:ThetaCov}). It is clear by definition that ${\rm supp}\zeta_{k,i}\subseteq B(\sigma_{i},2\theta_{k}^{t})$
and $\sum_{i}\zeta_{k,i}^{2}=1$. We claim that 
\begin{equation}
\big|\nabla_{\sigma}\zeta_{k,i}\big|\leqslant\frac{L_{1,d}}{\theta_{k}^{t}}\ \ \ \forall t,k,i\label{eq:ZetaBound}
\end{equation}
where $L_{1,d}$ is some universal constant. Indeed, denoting $\bar{\zeta}_{k}\triangleq\sqrt{\sum_{j}\chi_{k,j}^{2}}$
one has 
\[
\big|\nabla_{\sigma}\zeta_{k,i}\big|=\big|\frac{\nabla_{\sigma}\chi_{k,i}}{\bar{\zeta}_{k}}-\frac{\chi_{k,i}}{\bar{\zeta}_{k}^{3}}\sum_{j}\chi_{k,j}\nabla_{\sigma}\chi_{k,j}\big|\leqslant\big|\nabla_{\sigma}\chi_{k,i}\big|+\sum_{j}\big|\nabla_{\sigma}\chi_{k,j}\big|.
\]
According to Lemma \ref{lem:2ThetaCov}, for each fixed $\sigma\in\mathbb{S}^{d-1}$
there are at most $D_{d}$ (a universal constant) number of $j$'s
such that $\sigma\in B(\sigma_{j},2\theta_{k}^{t})$. As a result,
one concludes from Property (iii) of $\chi_{k,i}$ that 
\[
\big|\nabla_{\sigma}\zeta_{k,i}\big|\leqslant\frac{D_{d}+1}{\theta_{k}^{t}}.
\]

\noindent\textit{\uline{Step three: verification of required properties for \mbox{$\varphi_{k,i}(\rho,\sigma)$}
defined by (\mbox{\ref{eq:POU}})}}. 

\vspace{2mm}First of all, one has 
\[
\sum_{k,i}\varphi_{k,i}(\rho,\sigma)^{2}=\sum_{k}\eta_{k}(\rho)^{2}\sum_{i}\zeta_{k,i}(\sigma)^{2}=\sum_{k}\eta_{k}(\rho)^{2}\cdot1=1.
\]
Next, we estimate the support $\hat{P}_{k}^{t}(i)$ of $\varphi_{k,i}$.
Let $A,B\in\hat{P}_{k}^{t}(i)$ and let $A',B'$ be their closest
points in the annulus $A_{k}^{t}$ respectively. Note that $A_{\sigma}',B_{\sigma}'$
(the angular components of $A',B'$) are both contained in $B(\sigma_{i},2\theta_{k}^{t})$.
Let $\Gamma$ be the geodesic on $\mathbb{S}^{d-1}$ connecting $A'_{\sigma},B'_{\sigma}$.
Note that $\Gamma\subseteq B(\sigma_{i},2\theta_{k}^{t})$. We claim
that $\Gamma$ can only intersect at most $K_{1,d}$ (a universal
constant) number of $B(\sigma_{j},\theta_{k}^{t})$'s. In fact, suppose
that $\tau\in\Gamma\cap B(\sigma_{j},\theta_{k}^{t}).$ Then for any
$\sigma\in B(\sigma_{j},\theta_{k}^{t})$, one has (denoting $d_{\sigma}$
as the distance on $\mathbb{S}^{d-1}$)
\[
d_{\sigma}(\sigma,\sigma_{i})\leqslant d_{\sigma}(\sigma,\sigma_{j})+d(\sigma_{j},\tau)+d(\tau,\sigma_{i})\leqslant\theta_{k}^{t}+\theta_{k}^{t}+2\theta_{k}^{t}=4\theta_{k}^{t}.
\]
It follows that 
\[
\bigcup_{j:\Gamma\cap B(\sigma_{j},\theta_{k}^{t})\neq\emptyset}B\big(\sigma_{j},\frac{\theta_{k}^{t}}{2}\big)\subseteq B(\sigma_{i},4\theta_{k}^{t}).
\]
But the above union is a disjoint union by our construction. Therefore,
the total number of $B(\sigma_{j},\theta_{k}^{t})$'s intersecting
$\Gamma$ is bounded above by 
\[
\frac{{\rm vol}(B(\sigma_{i},4\theta_{k}^{t}))}{{\rm vol}(B(\sigma_{j},\theta_{k}^{t}/2))}\leqslant K_{1,d}.
\]
According to the diameter estimate (\ref{eq:DiamEst}), one thus concludes
that 
\[
d^{t}(A,B)\leqslant d^{t}(A',B')+K_{1,d}\cdot C_{4,d}r\leqslant(2+K_{1,d}C_{4,d})r=:K_{d}r.
\]
This shows that 
\[
{\rm Diam}\hat{P}_{k}^{t}(i)\leqslant K_{d}r
\]
and thus $\hat{P}_{k}^{t}(i)$ is contained in a geodesic ball of
radius $K_{d}r.$ 

It remains to estimate the potential $\Phi_{r}^{t}\triangleq\sum_{k,i}|\nabla^{t}\varphi_{k,i}|_{t}^{2}$.
First of all, by using the metric expression (\ref{eq:MetricT}) one
can write 
\begin{equation}
\big|\nabla^{t}\varphi_{k,i}(\rho,\sigma)\big|_{t}^{2}=\big(\eta'_{k}(\rho)\zeta_{k,i}(\sigma)\big)^{2}+\alpha(t)^{2}\sinh^{-2}(\alpha(t)\rho)\eta_{k}(\rho)^{2}\big|\nabla_{\sigma}\zeta_{k,i}(\sigma)\big|^{2}.\label{eq:GradDecomp}
\end{equation}
Since $\sum_{i}\zeta_{k,i}^{2}=1$ and each $\rho$ belongs to the
supports of at most two $\eta_{k}$'s, one sees from the estimate
(\ref{eq:EtaBound}) that 
\begin{equation}
\sum_{k,i}\big(\eta'_{k}(\rho)\zeta_{k,i}(\sigma)\big)^{2}=\sum_{k}(\eta'_{k}(\rho))^{2}\leqslant\frac{1}{2\varepsilon^{2}r^{2}}.\label{eq:GradPart1}
\end{equation}
We now estimate the second term on the right hand side of (\ref{eq:GradDecomp}).

Firstly, according to the bounds (\ref{eq:EstAng}) and (\ref{eq:ZetaBound})
one has 
\[
\big|\nabla_{\sigma}\zeta_{k,i}\big|\leqslant L_{2,d}\frac{\sinh\big(k\alpha(t)r\big)}{\sinh\big(\alpha(t)r/2\big)}.
\]
In addition, due to the support property of $\eta_{k}$ one can assume
that $\rho\in[(k-1-\varepsilon)r,(k+\varepsilon)r]$. It follows that
\begin{align}
 & \alpha(t)^{2}\sinh^{-2}(\alpha(t)\rho)\eta_{k}(\rho)^{2}\big|\nabla_{\sigma}\zeta_{k,i}(\sigma)\big|^{2}\nonumber \\
 & \leqslant L_{3,d}\alpha(t)^{2}\sinh^{-2}\big((k-1-\varepsilon)\alpha(t)r\big)\cdot\big(\frac{\sinh\big(k\alpha(t)r\big)}{\sinh\big(\alpha(t)r/2\big)}\big)^{2}\nonumber \\
 & =\frac{L_{3,d}}{r^{2}}\times\big[\alpha(t)r\cdot\sinh^{-1}\big((k-1-\varepsilon)\alpha(t)r\big)\cdot\frac{\sinh\big(k\alpha(t)r\big)}{\sinh\big(\alpha(t)r/2\big)}\big]^{2}.\label{eq:ADBd2}
\end{align}
Here one should assume $k\geqslant2$; the $k=1$ case can be treated
separately in the same manner and its discussion will be omitted. 

Consider the function 
\[
F_{k}(x)\triangleq x\cdot\sinh^{-1}\big((k-1-\varepsilon)x\big)\cdot\frac{\sinh\big(kx\big)}{\sinh\big(x/2\big)},\ \ \ k\geqslant2,\ x\in[0,1].
\]
In our context, $x=\alpha(t)r$ which can be assumed $\leqslant1$
provided that $t$ is sufficiently large (depending on $r$). We claim
that 
\begin{equation}
F_{k}(x)\leqslant C\ \ \ \forall k\geqslant2,x\in[0,1]\label{eq:ADBd3}
\end{equation}
with some universal constant $C.$ Indeed, simple algebra yields that
\begin{align}
F_{k}(x) & =\frac{x}{\sinh(x/2)}\cdot e^{(1+\varepsilon)x}\cdot\big(1+e^{-2kx}\cdot\frac{e^{2(1+\varepsilon)x}-1}{1-e^{-2(k-1-\varepsilon)x}}\big)\nonumber \\
 & \leqslant\frac{x}{\sinh(x/2)}\cdot e^{(1+\varepsilon)x}\cdot\big(1+\frac{e^{2(1+\varepsilon)x}-1}{1-e^{-2(1-\varepsilon)x}}\big),\label{eq:AngDerBound}
\end{align}
where we used $e^{-2kx}\leqslant1$ and $k\geqslant2$ to reach the
last inequality. It is clear that the right hand side of (\ref{eq:AngDerBound}),
as a function of $x\in[0,1]$, is uniformly bounded by some universal
constant $C$. 

By substituting (\ref{eq:ADBd3}) into (\ref{eq:ADBd2}), one obtains
that 
\[
\alpha(t)^{2}\sinh^{-2}(\alpha(t)\rho)\eta_{k}(\rho)^{2}\big|\nabla_{\sigma}\zeta_{k,i}(\sigma)\big|^{2}\leqslant\frac{L_{4,d}}{r^{2}}
\]
uniformly for all $r,t,\rho,\sigma,k,i$ provided that $t>T_{r}$
with suitable constant $T_{r}$ (the precise requirement here is that
$\alpha(t)r\leqslant1$). To conclude Property (ii) of the lemma,
it suffices to observe from Lemma \ref{lem:2ThetaCov} that for any
given $(\rho,\sigma)$, there are at most $2\times D_{d}$ (a universal
constant) pairs of $(k,i)$ such that $\eta_{k}(\rho)\nabla_{\sigma}\zeta_{k,i}(\sigma)\neq0$.
Therefore, 
\begin{equation}
\sum_{k,i}\alpha(t)^{2}\sinh^{-2}(\alpha(t)\rho)\eta_{k}(\rho)^{2}\big|\nabla_{\sigma}\zeta_{k,i}(\sigma)\big|^{2}\leqslant\frac{L_{5,d}}{r^{2}}.\label{eq:GradPart2}
\end{equation}
Combining the two estimates (\ref{eq:GradPart1}) and (\ref{eq:GradPart2}),
one arrives at 
\[
\Phi_{r}^{t}(\rho,\sigma)=\sum_{k,i}|\nabla^{t}\varphi_{k,i}(\rho,\sigma)|_{t}^{2}\leqslant\big(\frac{1}{2\varepsilon^{2}}+L_{5,d}\big)\frac{1}{r^{2}}
\]
for all $t>T_{r}$ and all $(\rho,\sigma)\in Q_{\tilde{R}(t)}^{t}$.
This proves Property (ii) of the lemma. 
\end{proof}

\subsection{\label{subsec:P2L1Est}From pointwise to $L^{1}$ bounds}

In order to make use of eigenvalue estimates, one needs a technical
step to control $u_{R(t)}(t,o)$ in terms of its spatial average
\[
(u_{R(t)}(t,\cdot),{\bf 1})\triangleq\int_{Q_{R(t)}}u_{R(t)}(t,x)dx.
\]
Here $u_{R(t)}$ is the solution to the localised PAM (\ref{eq:LocalPAM})
on the geodesic ball $Q_{R(t)}$ of radius $R(t)$ under curvature
$\kappa\equiv-1.$ The main estimate is summarised as follows.
\begin{lem}
\label{lem:P2A}There exists a universal constant $C>0,$ such that
\begin{equation}
\langle u_{R(t)}(t,o)\rangle\leqslant C(2+R(t))^{1/2}\big(\langle(u_{R(t)}(t,\cdot),{\bf 1})\rangle+|Q_{R(t)}|\big)+e^{-1}\cdot e^{H(t+1)}\label{eq:P2A}
\end{equation}
for all $t>1,$ where $|Q_{R(t)}|$ denotes the volume of $Q_{R(t)}$.
\end{lem}
\begin{proof}
We write
\begin{equation}
u_{R(t)}(t,o)=\mathbb{E}_{o}\big[\exp\big(\int_{0}^{t}\xi(W_{s})ds\big);W([0,t])\subseteq Q_{R(t)}\big]=:I(t)+J(t),\label{eq:IJDecomp}
\end{equation}
where $I(t),J(t)$ correspond to a further localisation on $\int_{0}^{1}\xi(W_{s})ds\leqslant1$
and $\int_{0}^{1}\xi(W_{s})ds>1$ respectively. 

\vspace{2mm}\noindent\textit{\uline{Estimation of \mbox{$I(t)$}}:} 

\vspace{2mm}By using the Markov
property, one has 
\begin{align*}
I(t) & \leqslant e\cdot\mathbb{E}_{o}\big[{\bf 1}_{\{\tau_{R(t)}>1)\}}\mathbb{E}_{W_{1}}\big[\exp\big(\int_{0}^{t-1}\xi(\tilde{W}_{s})ds\big);\tilde{\tau}_{R(t)}>t-1\big]\big]\\
 & =e\cdot\mathbb{E}_{o}\big[{\bf 1}_{\{\tau_{R(t)}>1)\}}u^{\xi}(t-1,W_{1})\big],
\end{align*}
where $\tilde{W}$ is an independent Brownian motion and $\tilde{\tau}$
is its corresponding exit time. Since $W_{1}\in Q_{R(t)}$ on the
event $\{\tau_{R(t)}>1\}$, the last expectation is bounded above
by 
\[
\int_{Q_{R(t)}}p(1,z)u^{\xi}(t-1,z)dz
\]
where $p(s,z)$ is the hyperbolic heat kernel. According to Lemma
\ref{lem:HKUpperGen}, it is easily seen that 
\[
p(1,z)\leqslant C(2+R(t))^{1/2}\ \ \ \forall z\in Q_{R(t)}.
\]
As a consequence, one obtains that 
\[
I(t)\leqslant Ce\cdot(2+R(t))^{1/2}\cdot\big(u_{R(t)}^{\xi}(t-1,\cdot),{\bf 1}\big)
\]
On the other hand, by the spectral decomposition $\{(\lambda_{k},e_{k})\}$
of $\Delta+\xi$ with Dirichlet boundary condition, one can write
\begin{align*}
\big(u_{R(t)}(t-1,\cdot),{\bf 1}\big) & =\sum_{k}e^{(t-1)\lambda_{k}}(e_{k},{\bf 1})^{2}\\
 & \leqslant\sum_{k:\lambda_{k}\geqslant0}e^{t\lambda_{k}}(e_{k},{\bf 1})^{2}+\sum_{k:\lambda_{k}<0}(e_{k},{\bf 1})^{2}\\
 & \leqslant\sum_{k}e^{t\lambda_{k}}(e_{k},{\bf 1})^{2}+\sum_{k}(e_{k},{\bf 1})^{2}\\
 & =\big(u_{R(t)}(t,\cdot),{\bf 1}\big)+|Q_{R(t)}|
\end{align*}
It follows that 
\begin{equation}
\langle I(t)\rangle\leqslant Ce\cdot(2+R(t))^{1/2}\cdot\big(\langle(u_{R(t)}(t,\cdot),{\bf 1})\rangle+|Q_{R(t)}|\big).\label{eq:IEst}
\end{equation}

\noindent\textit{\uline{Estimation of \mbox{$J(t)$}}:} 

\vspace{2mm}On the event $\{\int_{0}^{1}\xi(W_{s})ds>1\}$,
one has 
\begin{align*}
\exp\big(\int_{0}^{t}\xi(W_{s})ds\big) & \leqslant e^{-1}\cdot\exp\big(\int_{0}^{t}(1+{\bf 1}_{[0,1]}(s))\xi(W_{s})ds\big)\\
 & =e^{-1}\cdot\exp\big((t+1)\int_{0}^{t}\xi(W_{s})\frac{1+{\bf 1}_{[0,1]}(s)}{t+1}ds\big).
\end{align*}
According to Jensen's inequality, the right hand side is bounded above
by
\[
e^{-1}\cdot\int_{0}^{t}e^{(t+1)\xi(W_{s})}\frac{1+{\bf 1}_{[0,1]}(s)}{t+1}ds.
\]
It follows that 
\[
J(t)\leqslant e^{-1}\cdot\int_{0}^{t}\mathbb{E}_{o}\big[e^{(t+1)\xi(W_{s})}\big]\frac{1+{\bf 1}_{[0,1]}(s)}{t+1}ds
\]
and thus 
\begin{align}
\langle J(t)\rangle & \leqslant e^{-1}\cdot\int_{0}^{t}\langle\mathbb{E}_{o}\big[e^{(t+1)\xi(W_{s})}\big]\rangle\frac{1+{\bf 1}_{[0,1]}(s)}{t+1}ds=e^{-1}\cdot e^{H(t+1)}.\label{eq:JEst}
\end{align}
The desired estimate (\ref{eq:P2A}) follows from (\ref{eq:IEst})
and (\ref{eq:JEst}).
\end{proof}

\subsection{\label{subsec:PfL1Asym}Completing the proof of Theorem \ref{thm:UpperBd}}

We are now in a position to put all the previously developed ingredients
together to give a proof of Theorem \ref{thm:UpperBd}. The major
steps are respectively summarised in the following subsections. 

\subsubsection{Step 1: From $u$ to $u_{R(t)}$}

Throughout the rest, we fix a choice of $R(t)$ such that 
\[
t^{5/4}\ll R(t)\ll\beta(t)=t^{3/2}.
\]
According to Lemma \ref{lem:UpperLocal}, one has 
\begin{equation}
\underset{t\rightarrow\infty}{\overline{\lim}}\frac{1}{\beta(t)}\log\langle e^{-H(t)}u(t,o)\rangle\leqslant\underset{t\rightarrow\infty}{\overline{\lim}}\frac{1}{\beta(t)}\log\langle e^{-H(t)}u_{R(t)}(t,o)\rangle.\label{eq:S1}
\end{equation}

\subsubsection{Step 2: From $u_{R(t)}(t,o)$ to the spatial average of $u_{\tilde{R}(t)}^{\xi_{t}}(\beta(t),\cdot)$}

Next, we recall the scaling relation (cf. Lemma \ref{lem:ResPAM}
and (\ref{eq:RescaleLocalPAM})) that 
\[
\big(u_{R(t)}(t,\cdot),{\bf 1}\big)=\alpha(t)e^{H(t)}\big(u_{\tilde{R}(t)}^{\xi_{t}}(\beta(t),\cdot),{\bf 1}\big)_{t}.
\]
Here we set $\tilde{R}(t)\triangleq R(t)/\alpha(t)$, $u_{\tilde{R}(t)}^{\xi_{t}}$
is the solution to the localised PAM (\ref{eq:RescaleLocalPAM}) on
$Q_{\tilde{R}(t)}^{t}$ under curvature $\kappa_{t}\equiv-\alpha(t)^{2}$
with generator ${\cal L}^{t}+\xi_{t}$ and $(\cdot,\cdot)_{t}$ denotes
the $L^{2}$-inner product with respect to $d^{t}x$. According to
Lemma \ref{lem:P2A}, one obtains the following estimate: 
\begin{align}
\langle e^{-H(t)}u_{R(t)}(t,o)\rangle & \leqslant C(2+R(t))^{1/2}\big(\alpha(t)\langle\big(u_{\tilde{R}(t)}^{\xi_{t}}(\beta(t),\cdot),{\bf 1}\big)_{t}\rangle+|Q_{R(t)}|\big)\nonumber \\
 & \ \ \ +e^{-1}\cdot e^{t\sigma^{2}+\sigma^{2}/2},\label{eq:Xi2Xit}
\end{align}
where we used the explicit fact that $H(t)=t^{2}\sigma^{2}/2$ with
$\sigma^{2}=Q(0)$. Since $|Q_{R(t)}|=O(e^{(d-1)R(t)})$ and $R(t)\ll\beta(t)$,
one concludes from (\ref{eq:Xi2Xit}) that 
\begin{equation}
\underset{t\rightarrow\infty}{\overline{\lim}}\frac{1}{\beta(t)}\log\langle e^{-H(t)}u_{R(t)}(t,o)\rangle\leqslant\underset{t\rightarrow\infty}{\overline{\lim}}\frac{1}{\beta(t)}\log\big\langle\big(u_{\tilde{R}(t)}^{\xi_{t}}(\beta(t),\cdot),{\bf 1}\big)_{t}\big\rangle.\label{eq:S2}
\end{equation}

\subsubsection{Step 3: Sacrifice of potential}

Let $r>0$ be a given fixed parameter. We shall make use of the decomposition
(\ref{eq:QDecomp}) of the geodesic ball $Q_{\tilde{R}(t)}^{t}$ introduced
in Section \ref{subsec:HypPart} and the associated partition of unity
$\{\varphi_{k,i}\}$ constructed in Lemma \ref{lem:POU}. Recall from
Property (ii) of the lemma that 
\[
\Phi_{r}^{t}\triangleq\sum_{k,i}|\nabla^{t}\varphi_{k,i}|_{t}^{2}\leqslant\frac{L}{r^{2}}
\]
for all $t>T_{r}$. It follows that 
\begin{align*}
u_{\tilde{R}(t)}^{\xi_{t}}(\beta(t),x) & =\mathbb{E}_{x}\big[\exp\big(\int_{0}^{\beta(t)}\xi_{t}(W_{s}^{t})ds\big);W^{t}([0,\beta(t)])\subseteq Q_{\tilde{R}(t)}^{t}\big]\\
 & \leqslant e^{L\beta(t)/r^{2}}\mathbb{E}_{x}\big[\exp\big(\int_{0}^{\beta(t)}(\xi_{t}-\Phi_{r}^{t})(W_{s}^{t})ds\big);W^{t}([0,\beta(t)])\subseteq Q_{\tilde{R}(t)}^{t}\big]\\
 & =e^{L\beta(t)/r^{2}}u_{\tilde{R}(t)}^{\xi_{t}-\Phi_{r}^{t}}(\beta(t),x)
\end{align*}
and thus 
\begin{equation}
\big(u_{\tilde{R}(t)}^{\xi_{t}}(\beta(t),\cdot),{\bf 1}\big)_{t}\leqslant e^{L\beta(t)/r^{2}}\big(u_{\tilde{R}(t)}^{\xi_{t}-\Phi_{r}^{t}}(\beta(t),\cdot),{\bf 1}\big)_{t}\label{eq:S3}
\end{equation}
for all $t>T_{r}.$

\subsubsection{Step 4: Dominance by principal eigenvalue}

The lemma below controls the spatial average of $u_{\tilde{R}(t)}^{\xi_{t}-\Phi_{r}^{t}}(\beta(t),\cdot)$
in terms of the principal Dirichlet eigenvalue of ${\cal L}^{t}+\xi_{t}-\Phi_{r}^{t}$.
To ease notation, we just write $\eta_{t}\triangleq\xi_{t}-\Phi_{r}^{t}$.
\begin{lem}
\label{lem:AvEigen}Let $\{\lambda_{k}^{\eta_{t}},\phi_{k}\}$ be
the $L^{2}$-spectral decomposition of the operator ${\cal L}^{t}+\eta_{t}$
with Dirichlet boundary condition on $Q_{\tilde{R}(t)}^{t}$. Then
one has 
\begin{equation}
\langle\big(u_{\tilde{R}(t)}^{\eta_{t}}(\beta(t),\cdot),{\bf 1}\big)_{t}\rangle\leqslant{\rm vol}^{t}\big(Q_{\tilde{R}(t)}^{t}\big)\cdot\langle e^{\beta(t)\lambda_{1}^{\eta_{t}}}\rangle\label{eq:AvEigen}
\end{equation}
for all $t.$ 
\end{lem}
\begin{proof}
One has the $L^{2}$-expansion 
\begin{equation}
u_{\tilde{R}(t)}^{\eta_{t}}(\beta(t),\cdot)=\sum_{k=1}^{\infty}e^{\beta(t)\lambda_{k}^{\eta_{t}}}(\phi_{k},{\bf 1})_{t}\phi_{k}\leqslant e^{\beta(t)\lambda_{1}^{\eta_{t}}}\sum_{k=1}^{\infty}(\phi_{k},{\bf 1})_{t}\phi_{k}.\label{eq:AvEigenPf}
\end{equation}
After integrating (\ref{eq:AvEigenPf}) over $Q_{\tilde{R}(t)}^{t}$
and taking expectation $\langle\cdot\rangle$, one obtains the estimate
(\ref{eq:AvEigen}).
\end{proof}
It is straight forward to see that ${\rm vol}^{t}(Q_{\tilde{R}(t)}^{t})=O(e^{(d-1)R(t)})=o(e^{\beta(t)}).$
As a result of Lemma \ref{lem:AvEigen}, one obtains that 
\begin{equation}
\underset{t\rightarrow\infty}{\overline{\lim}}\frac{1}{\beta(t)}\log\big\langle\big(u_{\tilde{R}(t)}^{\eta_{t}}(\beta(t),\cdot),{\bf 1}\big)_{t}\big\rangle\leqslant\underset{t\rightarrow\infty}{\overline{\lim}}\frac{1}{\beta(t)}\log\langle e^{\beta(t)\lambda_{1}^{\eta_{t}}}\rangle.\label{eq:S4}
\end{equation}

\subsubsection{\label{subsec:DecomEigen}Step 5: Decomposition of principal eigenvalue}

We now estimate $\lambda_{1}^{\eta_{t}}$ in terms of the principal
eigenvalue on a ball of \textit{fixed} radius. Given a continuous
potential $V:Q_{\tilde{R}(t)}^{t}\rightarrow\mathbb{R}$ and a sub-domain
$D\subseteq Q_{\tilde{R}(t)}^{t}$ with piecewise smooth boundary,
we use $\lambda_{1}^{V}(D)$ to denote the principal Dirichlet eigenvalue
of the operator ${\cal L}^{t}+V$ on $L^{2}(D,d^{t}x).$ 

First of all, a direct application of Lemma \ref{eq:DEigenEst} shows
that 
\[
\lambda_{1}^{\xi_{t}-\Phi_{r}^{t}}(Q_{\tilde{R}(t)}^{t})\leqslant\max_{k,i}\lambda_{1}^{\xi_{t}}(\hat{P}_{k}^{t}(i)),
\]
where $\hat{P}_{k}^{t}(i)$ is the support of $\varphi_{k,i}$. Therefore,
one has
\begin{align}
\big\langle\exp\big(\beta(t)\lambda_{1}^{\xi_{t}-\Phi_{r}^{t}}(Q_{\tilde{R}(t)}^{t})\big)\big\rangle & \leqslant\big\langle\max_{k,i}\exp\big(\beta(t)\lambda_{1}^{\xi_{t}}(\hat{P}_{k}^{t}(i))\big)\big\rangle\nonumber \\
 & \leqslant\sum_{k,i}\big\langle\exp\big(\beta(t)\lambda_{1}^{\xi_{t}}(\hat{P}_{k}^{t}(i))\big)\big\rangle.\label{eq:EigenDecEst}
\end{align}
According to Property (i) of Lemma \ref{lem:POU}, $\hat{P}_{k}^{t}(i)$
is contained in a geodesic ball of radius $Kr$ on $\mathbb{H}_{\alpha(t)}^{d}$,
say \textbf{$B_{k}^{t}(i)$}. One then knows from (\ref{eq:Rayleigh})
that 
\[
\lambda_{1}^{\xi_{t}}(\hat{P}_{k}^{t}(i))\leqslant\lambda_{1}^{\xi_{t}}(B_{k}^{t}(i)).
\]
In addition, note that both ${\cal L}^{t}$ (as the hyperbolic Laplacian
on $\mathbb{H}_{\alpha(t)}^{d}$) and the distribution of $\xi_{t}$
(cf. Lemma \ref{lem:IsoInv}) are invariant under isometries. As a
result, one has
\[
\lambda_{1}^{\xi_{t}}(B_{k}^{t}(i))\stackrel{d}{=}\lambda_{1}^{\xi_{t}}(\Sigma_{Kr}).
\]
On the other hand, recall that the total number of $(k,i)$'s admit
an upper bound given by (\ref{eq:NumEst}). As a consequence, the
right hand side of (\ref{eq:EigenDecEst}) is bounded above by 
\[
\frac{\tilde{R}(t)}{r}\times\big(\frac{\sinh R(t)}{\sinh\alpha(t)r/2}\big)^{d-1}\times\big\langle\exp\big(\beta(t)\lambda_{1}^{\xi_{t}}(\Sigma_{Kr})\big)\big\rangle.
\]
It follows that 
\begin{equation}
\underset{t\rightarrow\infty}{\overline{\lim}}\frac{1}{\beta(t)}\log\big\langle\exp\big(\beta(t)\lambda_{1}^{\xi_{t}-\Phi_{r}^{t}}(Q_{\tilde{R}(t)}^{t})\big)\big\rangle\leqslant\underset{t\rightarrow\infty}{\overline{\lim}}\frac{1}{\beta(t)}\log\big\langle\exp\big(\beta(t)\lambda_{1}^{\xi_{t}}(\Sigma_{Kr})\big)\big\rangle.\label{eq:S5}
\end{equation}

\subsubsection{Step 6: Application of fixed-domain upper bound}

As the final step, we enlarge the right hand side of (\ref{eq:S5})
to
\[
\underset{t\rightarrow\infty}{\overline{\lim}}\frac{1}{\beta(t)}\log\big\langle\sum_{k=1}^{\infty}\exp\big(\beta(t)\lambda_{k}^{\xi_{t}}(\Sigma_{Kr})\big)\big\rangle
\]
where $\{\lambda_{k}^{\xi_{t}}(\Sigma_{Kr})\}$ is the entire sequence
of Dirichlet eigenvalues. According to Lemma \ref{lem:DomainUpper},
one concludes that the above quantity is bounded by $-\chi_{Kr}$
(cf. (\ref{eq:ChiR}) for its definition). Combining this with the
estimates (\ref{eq:S1}, \ref{eq:S2}, \ref{eq:S3}, \ref{eq:S4},
\ref{eq:S5}), one arrives at the following inequality: 
\[
\underset{t\rightarrow\infty}{\overline{\lim}}\frac{1}{\beta(t)}\log\langle e^{-H(t)}u(t,o)\rangle\leqslant-\chi_{Kr}+\frac{L}{r^{2}},
\]
where $K,L$ are universal constants depending only on $d$. The desired
upper bound (\ref{eq:UpperBd}) thus follows by taking $r\rightarrow\infty.$

\subsection{\label{subsec:FlucExp}Identification of fluctuation exponent}

Combing both Theorem \ref{thm:LowerBd} and Theorem \ref{thm:UpperBd},
it is readily seen that
\[
\chi'\triangleq\lim_{R\rightarrow\infty}\chi_{R}
\]
exists, where $\chi_{R}$ is the local exponent defined by (\ref{eq:ChiR}).
To identify the two exponents $\chi'=\chi$ (cf. (\ref{eq:FlucExp})
for the definition of $\chi$), it is enough to establish the following
fact. Recall that ${\cal S}_{{\rm eu}}(\cdot)$ is the Euclidean Donsker-Varadhan
functional defined by (\ref{eq:DVFunc}), which is viewed as a functional
on ${\cal P}_{c}(\Sigma)$. We also recall that $\lambda_{0}^{{\rm eu};f,R}$
is the principal Dirichlet eigenvalue of $-(\Delta_{{\rm eu}}+f)$
on $\Sigma_{R}$.

\begin{prop}
\label{prop:LegendDV}For each $R>0,$ one has
\[
{\cal S}_{{\rm eu}}(\mu)=\sup_{f\in C_{b}(\Sigma_{R})}\big\{\int_{\Sigma_{R}}fd\mu+\lambda_{0}^{{\rm eu};f,R}\big\}
\]
for any probability measure $\mu$ supported on $\Sigma_{R}$. 
\end{prop}
To prove Proposition \ref{prop:LegendDV}, we make use of an abstract
tool from convex analysis which we first recall (cf. \cite[Theorem 2.2.15]{DS89}).
\begin{lem}
\label{lem:ConvLem}Let $X$ be a locally convex, Hausdorff topological
(real) vector space and let $X^{*}$ be its topological dual. Suppose
that $F:X\rightarrow(-\infty,\infty]$ is a lower semi-continuous,
convex function. Define $G:X^{*}\rightarrow(-\infty,\infty]$ by 
\begin{equation}
G(\lambda)\triangleq\sup\big\{_{X^{*}}\langle\lambda,x\rangle_{X}-F(x):x\in X\big\},\ \ \ \lambda\in X^{*}.\label{eq:LegTranDV}
\end{equation}
If $F$ is not identically equal to $\infty,$ then one has 
\[
F(x)=\sup\big\{_{X^{*}}\langle\lambda,x\rangle_{X}-G(\lambda):\lambda\in X^{*}\big\},\ \ \ x\in X.
\]
\end{lem}
In our context, we choose $X\triangleq{\cal M}(\Sigma_{R})$ (the
space of finite signed-measures on $\Sigma_{R}$) and define $E\triangleq{\cal P}(\Sigma_{R})$
to be the subspace of probability measures on $\Sigma_{R}$. Recall
that the topology on $X$ is generated by (\ref{eq:Top}) and one
also has the basic facts given by Lemma \ref{lem:Riesz}. In order
to apply Lemma \ref{lem:ConvLem}, we first define the functional
$F$. This is just the Donsker-Varadhan functional extended to $X$,
i.e. 
\[
F(\mu)\triangleq\int_{\Sigma_{R}}|\nabla\phi|^{2}dx
\]
if $\mu$ is a probability measure supported on $\Sigma_{R}$, $\mu\ll dx$
and $\phi\triangleq\sqrt{\frac{d\mu}{dx}}\in H_{0}^{1}(\Sigma_{R})$;
otherwise we simply set $F(\mu)=+\infty$. We also define $G:X^{*}\rightarrow(-\infty,\infty]$
according to the relation (\ref{eq:LegTranDV}). Since $F$ is finite
only on $\phi\in H_{0}^{1}(\Sigma_{R})$ with $\|\phi\|_{L^{2}}=1$,
under the identification (\ref{eq:Riesz}) it is easy to check that
\[
G(f)=-\lambda_{0}^{{\rm eu};f,R}\ \ \ \forall f\in C_{b}(\Sigma_{R}).
\]
Consequently, in order to complete the proof of Proposition \ref{prop:LegendDV}
it is enough to verify the following fact explicitly.
\begin{lem}
$F$ is convex and lower semi-continuous.
\end{lem}
\begin{proof}
(i) \textit{Convexity}. Let $\mu,\nu\in X$ and $\theta\in[0,1].$
Define $\rho\triangleq(1-\theta)\mu+\theta\nu.$ We want to show that
\begin{equation}
F(\rho)\leqslant(1-\theta)F(\mu)+\theta F(\nu).\label{eq:Conv}
\end{equation}
One can assume without loss of generality (WLOG) that both $F(\mu),F(\nu)$
are finite. In other words, one has 
\[
\frac{d\mu}{dx}=\phi^{2},\ \frac{d\nu}{dx}=\psi^{2}
\]
where $\phi,\psi\in H_{0}^{1}(\Sigma_{R})$ with $\|\phi\|_{L^{2}}=\|\psi\|_{L^{2}}=1$.
In this case, $\rho$ is a probability measure and $\frac{d\rho}{dx}=\zeta^{2}$
with 
\[
\zeta\triangleq\sqrt{(1-\theta)\phi^{2}+\theta\psi^{2}}.
\]
Explicit calculation shows that 
\begin{align*}
\nabla\zeta & =(1-\theta)\frac{\phi}{\zeta}\nabla\phi+\theta\frac{\psi}{\zeta}\nabla\psi\\
 & =\sqrt{1-\theta}\cdot\frac{\sqrt{1-\theta}\phi}{\zeta}\nabla\phi+\sqrt{\theta}\cdot\frac{\sqrt{\theta}\psi}{\zeta}\nabla\psi,
\end{align*}
and one thus has 
\begin{equation}
|\nabla\zeta|^{2}\leqslant\big(\sqrt{1-\theta}\cdot\frac{\sqrt{1-\theta}\phi}{\zeta}|\nabla\phi|+\sqrt{\theta}\cdot\frac{\sqrt{\theta}\psi}{\zeta}|\nabla\psi|\big)^{2}.\label{eq:ConvEst}
\end{equation}

To proceed further, we recall the following elementary inequality:
\begin{equation}
(acx+bdy)^{2}\leqslant a^{2}x^{2}+b^{2}y^{2}\label{eq:EleIneq}
\end{equation}
for any nonnegative $a,b,c,d,x,y$ provided that $a^{2}+b^{2}=c^{2}+d^{2}=1$.
This follows easily by observing that the difference between the two
sides is a perfect square. By applying (\ref{eq:EleIneq}) to the
RHS of (\ref{eq:ConvEst}), one finds that 
\[
|\nabla\zeta|^{2}\leqslant(1-\theta)|\nabla\phi|^{2}+\theta|\nabla\psi|^{2}.
\]
The convexity property (\ref{eq:Conv}) follows by integration. 

\vspace{2mm}\noindent (ii) \textit{Lower semi-continuity}. Let $F(\mu)>y$
with given fixed $\mu\in X$ and $y\in\mathbb{R}$. We want to show
that there exists a neighbourhood $U$ of $\mu$ (with respect to
the topology $\tau$) such that 
\begin{equation}
F(\nu)>y\ \ \ \forall\nu\in U.\label{eq:LowerSC}
\end{equation}

\noindent\textit{\uline{Case (i): \mbox{$F(\mu)$} is finite}}. 

\vspace{2mm}In this
case, $\mu\in E$ and $\sqrt{\frac{d\mu}{dx}}=:\phi\in H_{0}^{1}(\Sigma_{R})$.
Since $F=\infty$ on $E^{c}$, it is enough to establish (\ref{eq:LowerSC})
when $F$ is restricted on $E$. Since the induced topology on $E$
(the weak topology) is metrisable, this is equivalent to showing that
\begin{equation}
F(\mu)=\int_{\Sigma}|\nabla\phi|^{2}dx\leqslant\underset{n\rightarrow\infty}{\underline{\lim}}F(\nu_{n})\label{eq:LSCSeq}
\end{equation}
for any sequence $E\ni\nu_{n}\rightarrow\mu$ weakly. Let $\nu_{n}$
be such a sequence. One may further assume that the right hand side
of (\ref{eq:LSCSeq}) is finite; otherwise the claim is trivial. In
this situation, let us just assume WLOG that 
\[
F(\nu_{n})=\int_{\Sigma_{R}}|\nabla\psi_{n}|^{2}dx\rightarrow r
\]
for some finite number $r$, where $\psi_{n}=\sqrt{\frac{d\nu_{n}}{dx}}\in H_{0}^{1}(\Sigma_{R})$.
Then $\{\psi_{n}\}$ is a bounded sequence in $H_{0}^{1}(\Sigma_{R})$.
According to Rellich's compactness theorem and the weak compactness
of $H_{0}^{1}(\Sigma_{R})$, $\psi_{n}$ contains an $L^{2}$-convergent
subsequence as well as an $H^{1}$-weakly convergent subsequence,
say WLOG $\psi_{n}\rightarrow\text{some }\psi\in H_{0}^{1}(\Sigma_{R})$
both strongly in $L^{2}$ and weakly in $H^{1}$. Define the probability
measure $d\nu\triangleq\psi^{2}dx.$ Then one has $\nu_{n}\rightarrow\nu$
weakly, thus implying that $\mu=\nu$ and $\phi=\psi$ a.e. Since
$\psi_{n}\rightarrow\psi=\phi$ weakly in $H^{1},$ one concludes
that 
\[
\|\phi\|_{H^{1}}\leqslant\underset{n\rightarrow\infty}{\underline{\lim}}\|\psi_{n}\|_{H^{1}}.
\]
This immediately gives (\ref{eq:LSCSeq}).

\vspace{2mm}\noindent\textit{\uline{Case (ii): \mbox{$F(\mu)$} is infinite}}. 

\vspace{2mm}If $\mu$
is not a probability measure, it is obvious that any $\nu$ near $\mu$
(with respect to the topology $\tau$) cannot be a probability measure
and one thus has $F(\nu)=\infty.$ Let us assume that $\mu$ is a
probability measure but $F(\mu)=\infty$. The situation is again reduced
to the sequential setting and let us assume that $E\ni\nu_{n}\rightarrow\mu$
weakly. We want to show that 
\[
\underset{n\rightarrow\infty}{\underline{\lim}}F(\nu_{n})=\infty.
\]
Assuming on the contrary that the above limit is finite, a similar
argument to Case (i) shows that $\sqrt{\frac{d\nu_{n}}{dx}}$ contains
a subsequence that is both convergent strongly in $L^{2}$ and weakly
in $H^{1}$, say to some $\phi\in H_{0}^{1}(\Sigma_{R})$. Since $\nu_{n}\rightarrow\mu$
weakly, one then knows that $\phi=\sqrt{\frac{d\mu}{dx}}$ which contradicts
the assumption $F(\mu)=\infty$.
\end{proof}
Finally, we identify the exponent $\chi$ explicitly. 
\begin{lem}
\label{lem:ChiValue}One has $\chi=\sqrt{Q''(0)/2}d.$
\end{lem}
\begin{proof}
Since both functionals $J$ and ${\cal S}_{{\rm eu}}$ essentially
live in the Euclidean space, the result follows directly from \cite[Equation (4.6)]{GK00}
with $\kappa=1$ and $\Sigma^{2}=-Q''(0){\rm Id}$ in their notation. 
\end{proof}

\section{\label{sec:LpAsym}The general moment asymptotics}

In this section, we complete the proof of Theorem \ref{thm:MainThm}
for the general moment asymptotics. The path from $L^{1}$ to $L^{p}$
asymptotics ($p>1$) is quite standard and is a straight forward adaptation
of argument in \cite{GK00}. The moral is that $\langle u(t,o)^{p}\rangle\approx\langle u(pt,o)\rangle$
is a good approximation at the logarithmic scale under consideration.
We first address the lower asymptotics. 

\begin{proof}[Proof of Theorem \ref{thm:MainThm}: Lower asymptotics]

Let $R>0$ be given fixed. The same Jensen- and H\"older-type estimates
leading to \cite[Equation 2.12]{GK00} show that 
\begin{equation}
\langle u(t,o)^{p}\rangle\geqslant\langle(u_{R\alpha(pt)}(pt,\cdot),{\bf 1})\rangle\times\Big(\frac{\langle(u_{R\alpha(pt)}(pt,\cdot),{\bf 1})\rangle}{\langle\sum_{k}e^{pt\lambda_{k}}\rangle}\Big)^{p}.\label{eq:LpLowPf}
\end{equation}
Here $u_{R\alpha(pt)}$ is the localised PAM on $(\Sigma_{R\alpha(pt)},g^{1})$
defined by (\ref{eq:LocalPAM}), $(\cdot,\cdot)$ is the $L^{2}$-inner
product defined by 
\[
(f,g)\triangleq\frac{1}{|\Sigma_{R\alpha(pt)}|}\int_{\Sigma_{R\alpha(pt)}}f(x)g(x)d^{1}x
\]
and $\{\lambda_{k}\}$ are the principal Dirichlet eigenvalues of
$\Delta+\xi$ on $(\Sigma_{R\alpha(pt)},g^{1})$. By a standard scaling
argument, it is readily checked that 
\begin{equation}
\beta(s)\lambda_{k}^{\xi_{s}}(\Sigma_{R})=s\lambda_{k}-H(s)\ \ \ \forall s>0,\label{eq:EigenScale}
\end{equation}
where $\{\lambda_{k}^{\xi_{t}}\}$ are the principal Dirichlet eigenvalues
of ${\cal L}^{s}+\xi_{s}$ on $(\Sigma_{R},g^{s})$. By substituting
(\ref{eq:EigenScale}) into (\ref{eq:LpLowPf}), one finds that 
\[
\langle u(t,o)^{p}\rangle\geqslant\langle(u_{R\alpha(pt)}(pt,\cdot),{\bf 1})\rangle\times\Big(\frac{e^{-H(pt)}\langle(u_{R\alpha(pt)}(pt,\cdot),{\bf 1})\rangle}{\langle\sum_{k}e^{\beta(pt)\lambda_{k}^{\xi_{pt}}(\Sigma_{R})}\rangle}\Big)^{p}.
\]
It is now clear from the lower $L^{1}$ asymptotics (\ref{eq:LowerBdC})
as well as the fixed domain upper asymptotics (\ref{eq:DomainUpper})
that 
\begin{align*}
\underset{t\rightarrow\infty}{\underline{\lim}}\frac{1}{\beta(pt)}\log\big(e^{-H(pt)}\langle u(t,o)^{p}\rangle\big)\geqslant & \underset{t\rightarrow\infty}{\underline{\lim}}\frac{1}{\beta(pt)}\log\big(e^{-H(pt)}\langle(u_{R\alpha(pt)}(pt,\cdot),{\bf 1})\rangle\big)\\
 & \ \ \ +p\underset{t\rightarrow\infty}{\underline{\lim}}\frac{1}{\beta(pt)}\log\big(e^{-H(pt)}\langle(u_{R\alpha(pt)}(pt,\cdot),{\bf 1})\rangle\big)\\
 & \ \ \ -p\underset{t\rightarrow\infty}{\overline{\lim}}\frac{1}{\beta(pt)}\log\big\langle\sum_{k}e^{\beta(pt)\lambda_{k}^{\xi_{pt}}(\Sigma_{R})}\big\rangle\\
\geqslant & -\chi_{R}+p(-\chi_{R}-(-\chi_{R}))=-\chi_{R}.
\end{align*}
Here we remark that although the lower asymptotics (\ref{eq:LowerBdC})
was stated for $\langle u(t,o)\rangle$, it actually holds for $\langle(u_{R\alpha(pt)}(pt,\cdot),{\bf 1})\rangle$
which is exactly how we proved (\ref{eq:LowerBdC}) (cf. (\ref{eq:LowerBdA})).
By letting $R\rightarrow\infty$, one concludes that 
\begin{equation}
\underset{t\rightarrow\infty}{\underline{\lim}}\frac{1}{\beta(pt)}\log\big(e^{-H(pt)}\langle u(t,o)^{p}\rangle\big)\geqslant-\chi,\label{eq:LpLow}
\end{equation}
which gives the desired lower asymptotics.

\end{proof}

Next, we address the upper $L^{p}$ asymptotics. 

\begin{proof}[Proof of Theorem \ref{thm:MainThm}: Upper bound]

First of all, with the lower asymptotics (\ref{eq:LpLow}) at hand
a simple adaptation of the proof of Lemma \ref{lem:UpperLocal} shows
that 
\begin{equation}
\lim_{t\rightarrow\infty}\frac{\langle u_{R(pt)}(t,o)^{p}\rangle}{\langle u(t,o)^{p}\rangle}=1\label{eq:LpLocal}
\end{equation}
provided that $R(pt)\gg t^{5/4}$ (and we fix such a choice of $R(pt)$).
The next adaptation required is Lemma \ref{lem:P2A} in Section \ref{subsec:P2L1Est}.
We continue to use the decomposition (\ref{eq:IJDecomp}) with $Q_{R(t)}$
replaced by $Q_{R(pt)}$. A standard Jensen- and H\"older-type argument
applied to the estimates (\ref{eq:IEst}) and (\ref{eq:JEst}) easily
gives that 
\[
I(t)^{p}\leqslant2^{p-1}(Ce)^{p}(2+R(pt))^{p/2}\big(|Q_{R(pt)}|^{p-1}(u_{R(pt)}(pt,\cdot),{\bf 1})+|Q_{R(pt)}|^{p}\big)
\]
and 
\[
\langle J(t)^{p}\rangle\leqslant e^{-p}e^{H(p(t+1))}=e^{H(pt)+O(t)}
\]
respectively. By substituting the above two estimates into the decomposition
(\ref{eq:IJDecomp}), one obtains that 
\begin{align}
\langle u_{R(pt)}(t,o)^{p}\rangle\leqslant & 2^{p-1}(\langle I(t)^{p}\rangle+\langle J(t)^{p}\rangle)\nonumber \\
\leqslant & 2^{p-1}\big[2^{p-1}(Ce)^{p}(2+R(pt))^{p/2}\big(|Q_{R(pt)}|^{p-1}\langle(u_{R(pt)}(pt,\cdot),{\bf 1})\rangle\nonumber \\
 & \ \ \ +|Q_{R(pt)}|^{p}\big)+e^{H(pt)+O(t)}\big].\label{eq:UpPf}
\end{align}
Recall from Section \ref{subsec:PfL1Asym} that we essentially proved
the following upper $L^{1}$ asymptotics:
\[
\underset{t\rightarrow\infty}{\overline{\lim}}\frac{1}{\beta(t)}\log\big(e^{-H(t)}\langle(u_{R(t)}(t,\cdot),{\bf 1})\rangle\big)\leqslant-\chi.
\]
By applying this to (\ref{eq:UpPf}) and also making use of (\ref{eq:LpLocal}),
one concludes that 
\[
\underset{t\rightarrow\infty}{\overline{\lim}}\frac{1}{\beta(pt)}\log\big(e^{-H(pt)}\langle u(t,o)^{p}\rangle\big)\leqslant-\chi.
\]
This gives the desired upper asymptotics. 

The proof of Theorem \ref{thm:MainThm} is now complete. 
\end{proof}

\section*{Acknowledgement}

XG is supported by ARC grant DE210101352. WX acknowledges the support from the Ministry of Science and Technology via the National Key R\&D Program of China (no.2023YFA1010102) and National Science Foundation China via the standard project grant (no.8200906145). Both authors would like to thank Stephen Muirhead for his valuable discussions and suggestions which have led to various improvements of the current work.


\begin{thebibliography}{Cha84}

\bibitem[BCH24]{BCH24}F. Baudoin, L. Chen, C. Huang, C. Ouyang, S. Tindel and J. Wang. Parabolic Anderson model in bounded domains of recurrent metric measure spaces. \textit{ArXiv preprint}, 2024.

\bibitem[BCO24]{BCO24}F. Baudoin, H. Chen and C. Ouyang. Moment estimates for the stochastic heat equation on Cartan-Hadamard manifolds. \textit{ArXiv preprint}, 2024.

\bibitem[BOT23]{BOT23}F. Baudoin, C. Ouyang, S. Tindel and J. Wang. Parabolic Anderson model on Heisenberg groups: the It\^o setting. \textit{J. Funct. Anal.} 285 (1) (2023): 1-44.

\bibitem[CM95]{CM95} R.A. Carmona and S.A. Molchanov. Stationary parabolic
Anderson model and intermittency. \textit{Probab. Theory Relat. Fields}, 102
(1995): 433--453.

\bibitem[Cha84]{Chavel84} I. Chavel. \textit{Eigenvalues in Riemannian
geometry}. Academic Press, 1984.

\bibitem[CO25]{CO25}H. Chen and C. Ouyang. Global geometry within an SPDE well-posedness problem. \textit{ArXiv preprint}, 2025. 

\bibitem[COV24]{COV24}L. Chen, C. Ouyang and W. Vickery. Parabolic Anderson model with colored noise on torus. \textit{ArXiv preprint}, 2023. To appear in \textit{Bernoulli}.

\bibitem[DM98]{DM98} E.B. Davies and N. Mandouvalos. Heat kernel bounds
on hyperbolic space and Kleinian groups. \textit{Proc. London Math.
Soc.} 57 (3) (1998): 182--208.

\bibitem[DZ09]{DZ09}A. Dembo and O. Zeitouni. \textit{Large deviations
techniques and applications}. 2nd Edition (corrected printing). Springer-Verlag,
2009.

\bibitem[DS89]{DS89} J.D. Deuschel and D.W. Stroock. \textit{Large
deviations}. Academic Press, 1989.

\bibitem[GK00]{GK00} J. G\"artner and W. K\"onig. Moment asymptotics
for the continuous parabolic Anderson model. \textit{Ann. Appl. Probab.}
10 (1) (2000): 192--217.

\bibitem[GKM00]{GKM00}J. G\"artner and W. K\"onig and S.A. Molchanov. Almost sure asymptotics for the continuous parabolic Anderson model. \textit{Probab. Theory Related Fields}
118 (2000): 547--573.

\bibitem[GM90]{GM90} J.G\"artner and S. A. Molchanov. Parabolic problems for the Anderson model. I: intermittency and related topics. \textit{Comm. Math. Phys.}, 132(3), (1990): 613-655. 

\bibitem[GX25]{GX25}X. Geng and W. Xu. Parabolic Anderson model in the hyperbolic space. Part II: quenched asymptotics. \textit{ArXiv preprint}, 2025.

\bibitem[KS88]{KS88}I. Karatzas and S.E. Shreve. \textit{Brownian
motion and stochastic calculus}. Springer-Verlag, 1988.

\bibitem[K\"on16]{K\"on16}W. K\"onig. \textit{The parabolic Anderson model}. Birkh\"auser, 2016.

\bibitem[RS78]{RS78}M. Reed and B. Simon. \textit{Methods of modern mathematical physics. IV. Analysis of operators}. Academic Press, New York-London, 1978.

\bibitem[TV02]{TV02}S. Tindel and F. Viens. Almost sure exponential behaviour for a parabolic SPDE on a manifold. \textit{Stoch. Process. Their Appl.} 100 (1-2) (2002): 53-74.

\end{thebibliography}
\end{document}